\documentclass[a4paper,11pt]{amsart}
\usepackage{amsmath}
\usepackage{amssymb}
\usepackage{amsthm}
\usepackage[english]{babel}
\usepackage{hyperref}
\usepackage{mathrsfs}

\usepackage{enumerate}

\newtheorem{theor}{Theorem}[section]
\newtheorem{lemma}[theor]{Lemma}

\newcommand{\hd}{\mathrm{hd}}
\newcommand{\soc}{\mathrm{soc}}
\newcommand{\rad}{\mathrm{rad}}
\newcommand{\Hom}{\mathrm{Hom}}
\newcommand{\End}{\mathrm{End}}
\newcommand{\Ext}{\mathrm{Ext}}

\newcommand{\res}{\mathrm{res}}
\newcommand{\s}{\Sigma}

\newcommand{\md}{\mbox{-}\mathrm{mod}}
\newcommand{\Md}{\!\mod}

\newcommand{\Z}{\mathbb{Z}}
\newcommand{\N}{\mathbb{N}}

\newcommand{\1}{\mathbf{1}}
\newcommand{\sgn}{\mathbf{\mathrm{sgn}}}

\renewcommand{\Im}{\mathrm{Im}}
\newcommand{\Ker}{\mathrm{Ker}}
\renewcommand{\epsilon}{\varepsilon}
\newcommand{\eps}{\epsilon}
\renewcommand{\phi}{\varphi}
\newcommand{\xymat}{\xymatrix@R=6pt@C=10pt}

\newcommand{\la}{\lambda}
\newcommand{\be}{\beta}
\newcommand{\al}{\alpha}
\newcommand{\ga}{\gamma}
\newcommand{\de}{\delta}

\newcommand{\Mull}{{\tt M}}

\newcommand{\da}{{\downarrow}}
\newcommand{\ua}{{\uparrow}}

\def\Par{{\mathscr P}}
\def\Parinv{{\mathscr P}^A}

\begin{document}

\title[Tensor products for alternating groups]{Irreducible tensor products for alternating groups in characteristics 2 and 3}

\author{\sc Lucia Morotti}
\address
{Institut f\"{u}r Algebra, Zahlentheorie und Diskrete Mathematik\\ Leibniz Universit\"{a}t Hannover\\ 30167 Hannover\\ Germany} 
\email{morotti@math.uni-hannover.de}

\begin{abstract}
In this paper we study irreducible tensor products of representations of alternating groups in characteristic 2 and 3. In characteristic 3 we completely classify irreducible tensor products, while in characteristic 2 we completely classify irreducible tensor products where neither factor in the product is a basic spin module. In characteristic 2 we also give some necessary conditions for the tensor product of an irreducible module with a basic spin module to be irreducible.
\end{abstract}

\maketitle

\section{Introduction}

Let $F$ be an algebraically closed field of characteristic $p\geq 0$ and $G$ be a group. In general, given irreducible $FG$-representations $V$ and $W$, the tensor product $V\!\otimes\! W$ is not irreducible. We say that $V\!\otimes\! W$ is a non-trivial irreducible tensor product if $V\!\otimes\! W$ is irreducible and neither $V$ nor $W$ has dimension 1. One motivation to this question is the Aschbacher-Scott classification of maximal subgroups of finite classical groups, see \cite{a} and \cite{as}. In particular, in view of class $\mathcal{C}_4$, a classification of non-trivial irreducible tensor products is needed to understand which subgroups appearing in class ${\mathcal S}$ are maximal, see \cite{Ma} for more details.

Non-trivial irreducible tensor products of representations of symmetric groups have been fully classified (see \cite{bk}, \cite{gj}, \cite{gk}, \cite{m1} and \cite{z1}). In particular non-trivial irreducible tensor products for symmetric groups only exist in characteristic 2 for $n\equiv 2\Md 4$. For alternating groups in characteristic 0 or $p\geq 5$ non-trivial irreducible tensor products have been classified in \cite{bk3}, \cite{bk2}, \cite{m2} and \cite{z1}.

In this paper we will consider the case where $G=A_n$ is an alternating group and $p=2$ or $3$. Our main result, which extends \cite[Main Theorem]{bk2} and \cite[Theorem 1.1]{m2} in a slight modified version, is the following. For an explanation of the notations used see \S\ref{sim} and the last part of \S\ref{sb}.

\begin{theor}\label{mt}
Let $V$ and $W$ be irreducible $FA_n$-modules of dimension larger than 1. If $V\otimes W$ is irreducible then one of the following holds up to exchange of $V$ and $W$:
\begin{enumerate}
\item $p\nmid n$, $V\cong E^\la_\pm$ where $\la$ is a JS-partition and $W\cong E^{(n-1,1)}$. In this case $V\otimes W$ is always irreducible and 
$V\otimes W\cong E^{(\la\setminus A)\cup B}$, where $A$ is the top removable node of $\la$ and $B$ is the second bottom addable node of $\la$.

\item $p=3$, $V\cong E^{(4,1^2)}_+$ and $W\cong E^{(4,1^2)}_-$. In this case $V\otimes W\cong E^{(4,2)}$.

\item $p=2$, $V$ is basic spin and at least one of $V$ or $W$ cannot be extended to a $F\s_n$-module.
\end{enumerate}
\end{theor}

Note that in the first two cases the tensor products are irreducible. This does however not always hold in the third case. A classification of irreducible tensor products with a basic spin module for alternating groups in characteristic 2 is currently not known. In Section \ref{s2} we will consider case (iii) more in details and give some conditions for such products to be irreducible.

In the next section we will give an overview of known results which will be used in the paper. In Section \ref{s1} as well as in Sections \ref{geq2n} to \ref{sJS} we study, in different ways, certain submodules of the modules $\Hom_{\s_n}(D^\la)$ and $\Hom_{A_n}(E^\la_\pm)$, using results from Sections \ref{sbr} and \ref{sph} and Section \ref{spm} respectively. These results will then be used in Sections \ref{sns} and \ref{ds} to study tensor products of a non-split and a split modules and of two split modules. Together with results on tensor products for modules of symmetric groups this will allow us to prove Theorem \ref{mt} in Section \ref{s3}. Although we cannot completely classify irreducible tensor products in characteristic 2 with a basic spin module, we will give some more restrictions for such tensor products to be irreducible in Section \ref{s2}.

\section{Notations and basic results}\label{snot}

Throughout the paper $F$ will be an algebraically closed field of characteristic $p$.

Given modules $M$ and $N_1,\ldots,N_h$ we will write
\[M\sim N_1|\ldots|N_h\]
if $M$ has a filtration with subquotients $N_j$ counted from the bottom and
\[M\sim (N_{1,1}|\ldots|N_{1,h_1})\,\,\oplus\,\,\ldots\,\,\oplus\,\, (N_{k,1}|\ldots|N_{k,h_k})\]
if there exists modules $M_i,N_{j,\ell}$ such that $M\cong M_1\oplus\ldots\oplus M_k$ and $M_j\sim N_{j,1}|\ldots|N_{j,h_j}$ for $1\leq j\leq k$. Further if modules $V_1,\ldots,V_h$ are simple, we will write
\[M\cong V_1|\ldots|V_h\]
if $M$ is uniserial with factors $V_j$ counted from the bottom and then similarly to above we will also write
\[M\cong (V_{1,1}|\ldots|V_{1,h_1})\,\,\oplus\,\,\ldots\,\,\oplus\,\, (V_{k,1}|\ldots|V_{k,h_k}).\]

For certain specific modules $V$, where $V$ is a simple or (dual of a) Specht module or direct sum of such, we will sometimes write $V\subseteq M$. When writing this we will always mean that $V$ is contained in $M$ up to isomorphism.

\subsection{Irreducible modules}\label{sim}

It is well known that irreducible representations of symmetric groups in characteristic $p$ are indexed by $p$-regular partitions and that they are self-dual. For $\la\in\Par_p(n)$ a $p$-regular partition, let $D^\la$ be the corresponding simple $F\s_n$-module. The module $D^\la$ can be defined as the head of $S^\la$, see \cite[Corollary 12.2]{JamesBook}. Further let $\la^\Mull\in\Par_p(n)$, the Mullineux dual of $\la$, be the unique partition with $D^{\la^\Mull}\cong D^\la\otimes \sgn$ (where $\sgn$ is the sign representation of $F\s_n$). 

For $p\geq 3$  it is well known that if $\lambda\not=\lambda^\Mull$ then $D^\lambda\da_{A_n}=E^\lambda$ is irreducible (and in this case $E^\lambda\cong E^{\lambda^\Mull}$), while if $\lambda=\lambda^\Mull$ then $D^\lambda\da_{A_n}=E^\lambda_+\oplus E^\lambda_-$ is the direct sum of two non-isomorphic irreducible representations of $A_n$. Further all irreducible representations of $A_n$ are of one of these two forms (see for example \cite{f}). If $p=2$ there is a different description of splitting irreducible representations (see Lemma \ref{split2}). Also in this case either $D^\la\da_{A_n}$ is irreducible or it is the direct sum of two non-isomorphic irreducible representations and any irreducible representation of $A_n$ is of one of these two forms.

For any $p$ let
\[\Parinv_p(n):=\{\la\in\Par_p(n)|D^\la\da_{A_n}\mbox{ splits}\}.\]
If $p\geq 3$ we have from the previous paragraph that $\la\in\Parinv_p(n)$ if and only if $\la=\la^\Mull$. For $p=2$ we have the following result:

\begin{lemma}{\cite[Theorem 1.1]{Benson}}\label{split2}
Let $p=2$ and $\la\in\Par_2(n)$. Then $\la\in\Parinv_2(n)$ if and only if the following hold
\begin{itemize}
\item $\la_{2i-1}-\la_{2i}\leq 2$ for each $i\geq 1$ and

\item $\la_{2i-1}+\la_{2i}\not\equiv 2\Md 4$ for each $i\geq 1$.
\end{itemize}
\end{lemma}

When considering splitting modules for $p\geq 3$ we have the following result, where $h(\lambda)$ is the number of parts of $\lambda$:

\begin{lemma}\label{Mull}{\cite[Lemma 1.8]{ks2}}
Let $p\geq 3$ and $n\geq 5$. If $\lambda\in\Parinv_p(n)$ then $h(\la)\geq 3$.
\end{lemma}

If $p=2$ a special role will be played by the irreducible modules indexed by the partition $\be_n:=(\lceil(n+1)/2\rceil,\lfloor (n-1)/2\rfloor)$. Such modules (for $\s_n$) can be obtained by reducing modulo 2 a basic spin module of the covering group of $\s_n$ and are therefore also called basic spin modules (see \cite{Benson}).

It easily follows from Lemmas \ref{split2} and \ref{Mull} that for large $n$, splitting modules cannot be indexed by partitions with at most two rows, unless possibly $p=2$ and the module is a basic spin module.

%

\subsection{Branching}\label{sb}

Since we will often study restrictions of modules to Young subgroups, we will now give a review of the needed branching results. 

Given a node $(a,b)$ define its residue by $\res(a,b)=b-a\Md p$. Given a partition $\la$ define its content to be the tuple $(c_0,\ldots,c_{p-1})$, where $c_i$ is the number of nodes of $\la$ of residue $i$, for each residue $i$. Two simple $F\s_n$-modules are in the same block if and only if the corresponding partitions have the same content. Thus we may define the content of a block and distinct blocks have distinct contents. For a residue $i$ and a module $M$ contained in the block with content $(c_0,\ldots,c_{p-1})$, let $e_iM$ (resp. $f_iM$) be the block component of $M\da_{\s_{n-1}}$ (resp. $M\ua^{\s_{n+1}}$) contained in the block with content $(c_0,\ldots,c_{i-1},c_i-1,c_{i+1},\ldots,c_{p-1})$ (resp. $(c_0,\ldots,c_{i-1},c_i+1,c_{i+1},\ldots,c_{p-1})$) if such a block exists or let $e_iM:=0$ (resp. $f_iM:=0$) otherwise. The definitions of $e_iM$ and $f_iM$ can then be extended to arbitrary modules additively. Then:

\begin{lemma}\label{l45}
For $M$ a $F\s_n$-module we have
\[M\da_{\s_{n-1}}\cong e_0M\oplus\ldots\oplus e_{p-1}M\hspace{18pt}\mbox{and}\hspace{18pt}M\ua^{\s_{n+1}}\cong f_0M\oplus\ldots\oplus f_{p-1}M.\]
\end{lemma}

\begin{proof}
We may assume that $M$ has only one block component. For $M$ simple the result holds by \cite[Theorems 11.2.7, 11.2.8]{KBook}. The result then hold in general by definition of $e_i$ and $f_i$ (there are no other block components).
\end{proof}

The following properties of $e_i$ and $f_i$ can be seen as special cases of \cite[Lemma 8.2.2]{KBook}.

\begin{lemma}\label{l57}
If $M$ is self dual then so are $e_iM$ and $f_iM$.
\end{lemma}

\begin{lemma}\label{l48}
The functors $e_i$ and $f_i$ are left and right adjoint of each other.
\end{lemma}

For $r\geq 1$ define $e_i^{(r)}:F\s_n\md\rightarrow F\s_{n-r}\md$ and $f_i^{(r)}:F\s_n\md\rightarrow F\s_{n+r}\md$ to be the divided power functors (see \cite[\S11.2]{KBook} for the definitions). For $r=0$ define $e_i^{(0)}D^\lambda$ and $f_i^{(0)}D^\lambda$ to be equal to $D^\lambda$. For a partition $\lambda$ let $\epsilon_i(\lambda)$ be the number of normal nodes of $\lambda$ of residue $i$ and $\phi_i(\lambda)$ be the number of conormal nodes of $\lambda$ of residue $i$ (see \cite[\S11.1]{KBook} or \cite[\S2]{bk2} for definitions of normal and conormal nodes). Normal and conormal nodes of partitions will play a crucial role through all of the paper. If $\epsilon_i(\lambda)\geq 1$ we will denote by $\tilde{e}_i(\lambda)$ the partition obtained from $\lambda$ by removing the $i$-good node, that is the bottom $i$-normal node. Similarly, if $\phi_i(\lambda)\geq 1$ we denote by $\tilde{f}_i(\lambda)$ the partition obtained from $\lambda$ by adding the $i$-cogood node, that is the top $i$-conormal node. The next two lemmas will be used throughout the paper and show that the modules $e_i^rD^\lambda$ and $e_i^{(r)}D^\lambda$ (and similarly $f_i^rD^\lambda$ and $f_i^{(r)}D^\lambda$) are closely connected. For $r=0$ the lemmas hold trivially. For $r>0$ see \cite[Theorems 11.2.10, 11.2.11]{KBook}.

\begin{lemma}\label{l39}
Let $\lambda\in\Par_p(n)$, $r\geq 0$ and $i$ be a residue. Then $e_i^rD^\lambda\cong(e_i^{(r)}D^\lambda)^{\oplus r!}$. Further $e_i^{(r)}D^\lambda\not=0$ if and only if $\epsilon_i(\lambda)\geq r$. In this case
\begin{enumerate}
\item\label{l39a}
$e_i^{(r)}D^\lambda$ is a self-dual indecomposable module with head and socle isomorphic to $D^{\tilde{e}_i^r(\lambda)}$,

\item\label{l39b}
$[e_i^{(r)}D^\lambda:D^{\tilde{e}_i^r(\lambda)}]=\binom{\epsilon_i(\lambda)}{r}=\dim\End_{\s_{n-r}}(e_i^{(r)}D^\lambda)$,

\item\label{l39c}
if $D^\psi$ is a composition factor of $e_i^{(r)}D^\lambda$ then $\epsilon_i(\psi)\leq \epsilon_i(\lambda)-r$, with equality holding if and only if $\psi=\tilde{e}_i^r(\lambda)$.
\end{enumerate}
\end{lemma}

\begin{lemma}\label{l40}
Let $\lambda\in\Par_p(n)$, $r\geq 0$ and $i$ be a residue. Then $f_i^rD^\lambda\cong(f_i^{(r)}D^\lambda)^{\oplus r!}$. Further $f_i^{(r)}D^\lambda\not=0$ if and only if $\phi_i(\lambda)\geq r$. In this case
\begin{enumerate}
\item\label{l40a}
$f_i^{(r)}D^\lambda$ is a self-dual indecomposable module with head and socle isomorphic to $D^{\tilde{f}_i^r(\lambda)}$,

\item\label{l40b}
$[f_i^{(r)}D^\lambda:D^{\tilde{f}_i^r(\lambda)}]=\binom{\phi_i(\lambda)}{r}=\dim\End_{\s_{n+r}}(f_i^{(r)}D^\lambda)$,

\item\label{l40c}
if $D^\psi$ is a composition factor of $f_i^{(r)}D^\lambda$ then $\phi_i(\psi)\leq \phi_i(\lambda)-r$, with equality holding if and only if $\psi=\tilde{f}_i^r(\lambda)$.
\end{enumerate}
\end{lemma}

For $r=1$ it follows that $e_i=e_i^{(1)}$ and $f_i=f_i^{(1)}$. In this case more composition factors of $e_iD^\lambda$ and $f_iD^\lambda$ are known by \cite[Theorem E(iv)]{bk6} and \cite[Theorem 1.4]{k4}.

\begin{lemma}\label{l56}
Let $\lambda\in\Par_p(n)$. If $A$ is an $i$-normal node of $\lambda$ and $\lambda\setminus A$ is $p$-regular then $[e_iD^\lambda:D^{\lambda\setminus A}]$ is equal to the number of $i$-normal nodes of $\lambda$ weakly above $A$.

Similarly if $B$ is an $i$-conormal node of $\lambda$ and $\lambda\cup B$ is $p$-regular then $[f_iD^\lambda:D^{\lambda\cup B}]$ is equal to the number of $i$-conormal nodes of $\lambda$ weakly below $B$.
\end{lemma}

Since the modules $e_iD^\lambda$ (or $f_iD^\lambda$) correspond to pairwise distinct blocks, the following holds combining Lemmas \ref{l45}, \ref{l39}(ii) and \ref{l40}(ii).

\begin{lemma}\label{l53}
For $\la\in\Par_p(n)$ we have that
\begin{align*}
\dim\End_{\s_{n-1}}(D^\lambda\da_{\s_{n-1}})&=\epsilon_0(\lambda)+\ldots+\epsilon_{p-1}(\lambda),\\
\dim\End_{\s_{n+1}}(D^\lambda\ua^{\s_{n+1}})&=\phi_0(\lambda)+\ldots+\phi_{p-1}(\lambda).
\end{align*}
\end{lemma}

When considering the functors $\tilde{e}_i$ and $\tilde{f}_i$ the following easily holds by definition (alternatively see \cite[Lemma 5.2.3]{KBook} for the first part and Lemmas \ref{l39}(iii) and \ref{l40}(iii) for the second part).

\begin{lemma}\label{l47}
For $r\geq 0$ and $p$-regular partitions $\lambda,\nu$ we have that $\tilde{e}^r_i(\lambda)=\nu$ if and only if $\tilde{f}^r_i(\nu)=\lambda$. In this case $\epsilon_i(\nu)=\epsilon_i(\lambda)-r$ and $\phi_i(\nu)=\phi_i(\lambda)+r$.
\end{lemma}

The total numbers of normal and conormal nodes of a partition are related by following result, which hold by the corresponding result for removable and addable nodes and by definition of normal and conormal nodes (the set of normal and conormal nodes is obtained by recursively removing pairs of a removable and an addable node from the set of removable and addable nodes).

\begin{lemma}\label{l52}
Any $p$-regular partition has 1 more conormal node than it has normal nodes.
\end{lemma}

The following result connects branching and the Mullineux bijection (see \cite[Theorem 
4.7]{kMull} or \cite[Lemma 5.10]{m2}).

\begin{lemma}\label{l17}
For any partition $\lambda\in\Par_p(n)$ and for any residue $i$ we have $\epsilon_i(\lambda)=\epsilon_{-i}(\lambda^\Mull)$ and $\phi_i(\lambda)=\phi_{-i}(\lambda^\Mull)$.

If $\epsilon_i(\lambda)>0$ then $\tilde{e}_i(\lambda)^\Mull=\tilde{e}_{-i}(\lambda^\Mull)$, while if $\phi_i(\lambda)>0$ then $\tilde{f}_i(\lambda)^\Mull=\tilde{f}_{-i}(\lambda^\Mull)$.
\end{lemma}

We conclude by defining JS-partitions. A JS-partition is a partition $\la\in\Par_p(n)$ for which $D^\lambda\da_{\s_{n-1}}$ is irreducible. In view of Lemmas \ref{l45} and \ref{l39} a $p$-regular partition is a JS-partition if and only if it has exactly one normal node. By Lemma \ref{l52} we then also have that JS-partitions have exactly 2 conormal nodes. JS-partitions will play a special role in this paper. They have a nice combinatorial description, see \cite[Section 4]{JS} and  \cite[Theorem D]{k2}:

\begin{lemma}\label{L221119}
Let $\lambda=(a_1^{b_1},\ldots,a_h^{b_h})$ with $a_1>a_2>\ldots>a_h\geq 1$ and $1\leq b_i\leq p-1$ for $1\leq i\leq h$. Then $\lambda$ is a JS-partition if and only if $a_i-a_{i+1}+b_i+b_{i+1}\equiv 0\Md p$ for each $1\leq i<h$.
\end{lemma}

For $p=2$ this simplifies to:

\begin{lemma}\label{L151119}
Let $p=2$ and $\lambda\in\Par_2(n)$. Then $\lambda$ is a JS-partition if and only if all parts of $\la$ have the same parity.
\end{lemma}

\subsection{Permutation modules}

For any composition $\la$ of $n$ let $\s_\la=\s_{\la_1}\times\s_{\la_2}\times\ldots\subseteq\s_n$ be the corresponding Young subgroup and define $M^\la:=\1\ua_{\s_\la}^{\s_n}$. Clearly the modules $M^\la$ are self-dual, as is any permutation module. Note that if $\la$ and $\mu$ can be obtained from each other by rearranging their parts, then $M^\la\cong M^\mu$. So from now on we will assume that $\la\in\Par(n)$ is a partition. In this case let $S^\la$ be the Specht module indexed by $\la$. It is well known that $S^\la\subseteq M^\la$ (this holds for example by comparing standard bases of $M^\la$ and $S^\la$). Further let $Y^\la$ be the corresponding Young module, that is the module given by the following lemma (see \cite{JamesArcata} and \cite[\S4.6]{Martin}). In the lemma $\rhd$ denotes the dominance order.

\begin{lemma}{} \label{LYoung}
There exist indecomposable $F \s_n$-modules $\{Y^\lambda\mid \lambda\in\Par(n)\}$ such that $M^\lambda\cong Y^\lambda\,\oplus\, \bigoplus_{\mu\rhd\lambda}(Y^\mu)^{\oplus m_{\mu,\lambda}}$ for some $m_{\mu,\lambda}\geq 0$. Moreover, $Y^\lambda$ can be characterized as the unique direct summand of $M^\lambda$ containing  $S^\lambda$. Further $Y^\lambda$ is self-dual for any $\lambda\in\Par(n)$.
\end{lemma}

The above lemma will be used in Section \ref{spm} to study the structure of certain small permutation modules. The structure of such permutation modules, together with the next lemma, will then be used in Sections \ref{geq2n} to \ref{sJS} to study the submodule structure of $\End_F(V)$, for $V$ a simple $F\s_n$- or $F A_n$-module using the next lemma, which holds by Frobenious reciprocity. For any partition $\al\in\Par(n)$ let $A_\al:=A_n\cap\s_\al$. It is easy, by Mackey induction-reduction theorem, to check that if $\al\not=(1^n)$ then $M^\al\da_{A_n}\cong\1\ua_{A_\al}^{A_n}$. 

\begin{lemma}\label{l2}
For any $F\s_n$-module $V$ and any $\alpha\in\Par(n)$ we have that
\[\dim\Hom_{\s_n}(M^\alpha,\End_F(V))=\dim\End_{\s_\alpha}(V\da_{\s_\alpha}).\]
Similarly for any $F A_n$-module $W$ and $\al\not=(1^n)$ we have that
\[\dim\Hom_{A_n}(M^\alpha,\End_F(V))=\dim\End_{A_\alpha}(V\da_{A_\alpha}).\]
\end{lemma}

The following lemma will play a crucial role in Sections \ref{sns} and \ref{ds} to prove that in most cases $V\otimes W$ is not simple, see \cite[Lemma 5.3]{m2} for a proof (for $p=2$ and $H=A_n$ the proof is similar).

\begin{lemma}\label{l15}
Let $H=\s_n$ or $H=A_n$ and let $V$ and $W$ be $FH$-modules. For $\alpha\in\Par(n)$ 
let $m_{V^*,\alpha},m_{W,\alpha}\in\Z_{\geq 0}$ be such that there exist $\phi^\alpha_1,\ldots,\phi^\alpha_{m_{V^*,\alpha}}\in\Hom_H(M^\alpha,V^*)$ with $\phi^\alpha_1|_{S^\alpha},\ldots,\phi^\alpha_{m_{V^*,\alpha}}|_{S^\alpha}$ linearly independent and similarly there exist $\psi^\alpha_1,\ldots,\psi^\alpha_{m_{W,\alpha}}\in\Hom_H(M^\alpha,W)$ with $\psi^\alpha_1|_{S^\alpha},\ldots,\psi^\alpha_{m_{W,\alpha}}|_{S^\alpha}$ linearly independent. If $H=\s_n$ let $A:=\Par_p(n)$. If $H=A_n$ and $p=2$ let $A:=\Par_2(n)\setminus\Parinv_2(n)$. If $H=A_n$ and $p\geq 3$ let $A$ be the set of partitions $\al\in\Par_p(n)\setminus\Parinv_p(n)$ with $\al>\al^\Mull$. Then
\[\dim\Hom_H(V,W)\geq\sum_{\alpha\in A}
m_{V^*,\alpha}m_{W,\alpha}.\]
\end{lemma}

Since we will often work with permutation modules $M^\la$ with $\la=(n-m,\mu)=(n-m,\mu_1,\mu_2,\ldots)$ for certain fixed partitions $\mu\in\Par(m)$ with $m$ small, we will write $M_\mu$, $S_\mu$ and $Y_\mu$ for $M^{(n-m,\mu)}$, $S^{(n-m,\mu)}$ and $Y^{(n-m,\mu)}$ respectively, provided $(n-m,\mu)\in\Par(n)$. Similarly, if $(n-m,\mu)\in\Par_p(n)$ is $p$-regular, we will write $D_\mu$ for the simple module $D^{(n-m,\mu)}$.

\section{Branching recognition}\label{sbr}

In this section we will show that under certain assumptions on $n$, if $\la\in\Par_p(n)$ is of certain special forms, then (some) restrictions $D^\la\da_{\s_{n-m}}$ have composition factors indexed by partitions with similar forms as $\la$. These results will be used in Section \ref{s1} to show existence of homomorphisms $M_\mu\to\End_F(D)$ which do not vanish on $S_\mu$ (for certain small partitions $\mu$), where $D$ is a simple $F\s_n$- or $F A_n$-module.

\begin{lemma}\label{L5}
Let $p=2$ and $\la\in\Par_2(n)$. If $n>h(\la)(h(\la)+1)/2$ there exists a composition factor $D^\mu$ of $D^\la\da_{\s_{n-1}}$ with $h(\mu)=h(\la)$. In particular $D^{(h(\la),h(\la)-1,\ldots,1)}$ is a composition factor of $D^\la\da_{\s_{h(\la)(h(\la)+1)/2}}$.
\end{lemma}

\begin{proof}
Note that for any $\la\in\Par_2(n)$ we have that $n\geq h(\la)(h(\la)+1)/2$, with equality holding if and only if $\la=(h(\la),h(\la)-1,\ldots,1)$. So the second part of the lemma follows from the first. Assume now that $n>h(\la)(h(\la)+1)/2$. Let $1\leq k\leq h(\la)$ minimal such that $\la_k\geq\la_{k+1}+2$ (such a $k$ exists since $n>h(\la)(h(\la)+1)/2$). Then $(k,\la_k)$ is normal and $\mu=\la\setminus(k,\la_k)\in\Par_2(n-1)$ with $h(\mu)=h(\la)$. The lemma then follows from Lemma \ref{l56}.
\end{proof}

\begin{lemma}\label{l34}
Let $p=3$, $n\geq 9$ and $\la\in\Par_3(n)$. If $h(\la),h(\la^\Mull)\geq 4$ then $D^\la\da_{\s_{n-1}}$ has a composition factor $D^\mu$ with $h(\mu),h(\mu^\Mull)\geq 4$.
\end{lemma}

\begin{proof}
Assume first that $h(\la),h(\la^\Mull)\geq 5$ and let $A$ be a good node of $\la$. Then $(\la\setminus A)^\Mull=\la^\Mull\setminus B$ for a good node $B$ of $\la^\Mull$ (see Lemma \ref{l17}). Then $D^{\la\setminus A}$ is a composition factor of $D^\la\da_{\s_{n-1}}$ by Lemma \ref{l39} and $h(\la\setminus A),h((\la\setminus A)^\Mull)\geq 4$. So, up to exchange of $\la$ and $\la^\Mull$ we may assume that $h(\la^\Mull)\geq h(\la)=4$. For any partition $\alpha$ let $G_1(\alpha)$ be the first column of the Mullineux symbol of $\alpha$ (see \cite[Section 1]{FK} for definition of the Mullineux symbol of $\la$ and how to obtain the Mullineux symbol of $\la^\Mull$ from that of $\la$). If $\la$ has a normal node $C$ such that $\la\setminus C$ is 3-regular and $G_1(\la)=G_1(\la\setminus C)$ then the lemma holds, by Lemma \ref{l56} and the combinatorial definition of $\la^\Mull$.

{\sf Case 1.} $\la_1=\la_2$. Then $\la_2>\la_3$ and we can take $C=(2,\la_2)$.

{\sf Case 2.} $\la_1=\la_2+1=\la_3+1=\la_4+2$. In this case $\la^\Mull=(2\la_1-1,2\la_1-3)$ by \cite[Lemma 2.3]{bkz}, contradicting the assumptions.

{\sf Case 3.} $\la_1=\la_2+1=\la_3+1=\la_4+3$. If $\la_1=4$ then $\la=(4,3,3,1)$ and $D^{(5,2,2,1)}$ is a composition factor of $D^{(4,3,3,1)}\da_{\s_{10}}$ by \cite[Tables]{JamesBook}. Since $h((5,2,2,1))=h((5,2,2,1)^\Mull)=4$, we may assume that $\la_1\geq 5$. In this case $D^{(\la_1,\la_2,\la_3,\la_4-1)}$ is a composition factor of $D^\la\da_{\s_{n-1}}$ and $h((\la_1,\la_2,\la_3,\la_4-1)),h((\la_1,\la_2,\la_3,\la_4-1)^\Mull)\geq4$.

{\sf Case 4.} $\la_1=\la_2+1=\la_3+1>\la_4+3$. In this case we can take $C=(3,\la_3)$.

{\sf Case 5.} $\la_1=\la_2+1>\la_3+1$. In this case we can take $C=(1,\la_1)$.

{\sf Case 6.} $\la_1=\la_2+2=\la_3+2$. Then $\la_3>\la_4$ and we can take $C=(3,\la_3)$.

{\sf Case 7.} $\la_1=\la_2+2=\la_3+3=\la_4+3$. If $\la_1=4$ then $n=8$, so we may assume that $\la_1\geq 5$. In this case $D^{(\la_1,\la_2,\la_3,\la_4-1)}$ is a composition factor of $D^\la\da_{\s_{n-1}}$ and $h((\la_1,\la_2,\la_3,\la_4-1)),h((\la_1,\la_2,\la_3,\la_4-1)^\Mull)\geq4$.

{\sf Case 8.} $\la_1=\la_2+2=\la_3+3\geq\la_4+4$. In this case we can take $C=(2,\la_2)$.

{\sf Case 9.} $\la_1=\la_2+2=\la_3+4=\la_4+4$.  If $\la_1\geq 6$ we can take $C=(4,\la_4)$. If $\la_1=5$ then $\la=(5,3,1,1)$, $D^{(5,2,1,1)}$ is a composition factor of $D^{(5,3,1,1)}$ and $h((5,2,1,1)),h((5,2,1,1)^\Mull)=4$.

{\sf Case 10.} $\la_1=\la_2+2=\la_3+4>\la_4+4$. In this case $D^{(\la_1,\la_2,\la_3-1,\la_4)}$ is a composition factor of $D^\la\da_{\s_{n-1}}$ and $h((\la_1,\la_2,\la_3-1,\la_4)),h((\la_1,\la_2,\la_3-1,\la_4)^\Mull)\geq4$.

{\sf Case 11.} $\la_1=\la_2+2\geq \la_3+5$. In this case we can take $C=(2,\la_2)$.

{\sf Case 12.} $\la_1\geq\la_2+3$. In this case we can take $C=(1,\la_1)$.
\end{proof}

\begin{lemma}\label{l13}
Let $p=3$, $n\geq 7$ and $\la=(n-k,k)$ with $n-2k\geq 2$ and $k\geq 2$. Then $D^\la\da_{\s_{n-1}}$ has a composition factor $D^\mu$ with $\mu=(n-1-\ell,\ell)$ with $n-1-2\ell\geq 2$ and $\ell\geq 2$. In particular $D^{(4,2)}$ is a composition factor of $D^\la\da_{\s_6}$.
\end{lemma}

\begin{proof}
If $n-2k\geq 3$ then we can take $\mu=(n-k-1,k)$ by Lemma \ref{l56}. If $n-2k=2$ then $k\geq 3$ since $n\geq 7$ and, again by Lemma \ref{l56}, we can take $\mu=(n-k,k-1)$. The result for $D^\la\da_{\s_6}$ follows by induction.
\end{proof}

\begin{lemma}\label{l10}
Let $p=3$, $n>6$ and $\la\in\Parinv_3(n)$ be a JS-partition with $h(\la)=3$. Then $n\equiv 0\Md 6$ and $D^\la\da_{\s_{n-6}}$ has a composition factor $D^\mu$ with $\mu\in\Parinv_3(n-6)$ a JS-partition with $h(\mu)=3$. Further $D^{(5,1^2)}$ is a composition factor of $D^\la\da_{\s_7}$.
\end{lemma}

\begin{proof}
We have that $\la\in\Parinv_3(m)$ if and only if $\la\in\Par_3(m)$ is Mullineux-fixed. From \cite[Theorem 4.1]{bo} we have that Mullineux-fixed partitions with 3-parts are exactly the partitions with Mullineux symbols
\[\left(\begin{array}{ccccc}6&\ldots&6&5&1\\3&\ldots&3&3&1\end{array}\right).\]
So $\la=\la^\Mull$ and $h(\la)=3$ if and only if $n\equiv 0\Md 6$ and $\la=(n/2+1,(n/2-1)^\Mull)$. Assume that this is the case. From Lemma \ref{l56} it follows that $D^{(n/2,(n/2-1)^\Mull)}$, $D^{(n/2,(n/2-2)^\Mull)}$, $D^{(n/2,(n/2-3)^\Mull)}$, $D^{(n/2-1,(n/2-3)^\Mull)}$, $D^{(n/2-1,(n/2-4)^\Mull)}$ and $D^{(n/2-2,(n/2-4)^\Mull)}$ are composition factors of $D^\la\da_{\s_{n-k}}$ with $1\leq k\leq 6$. We can then take $\mu=(n/2-2,(n/2-4)^\Mull)=((n-6)/2+1,((n-6)/2-1)^\Mull)$. For the last claim note that by induction $D^{(7,3,2)}$ is a composition factor of $D^\la\da_{\s_{12}}$ and by the previous $D^{(7,3,2)}\da_{\s_7}$ has a composition factor $D^{(5,1^2)}$.
\end{proof}

\section{Special homomorphisms}\label{sph}

In this section 
we will give conditions under which there exist homomorphisms $M_\mu\to\End_F(D)$ which do not vanish on $S_\mu$. Such conditions will then be checked to hold in some cases in the next section.

\begin{lemma}\label{l18}
Let $n\geq 6$ and $V$ be a $FA_n$-module. For pairwise distinct $a,b,c$ define $[a,b,c]:=(a,b,c)+(a,c,b)$. If
\[x_3:=[1,2,3]+[1,5,6]+[2,4,6]+[3,4,5]-[1,2,6]-[1,3,5]-[2,3,4]-[4,5,6]\]
and $x_3V\not=0$ then there exists $\psi\in\Hom_{A_n}(M_3\da_{A_n},\End_F(V))$ which does not vanish on $S_3\da_{A_n}$.
\end{lemma}

\begin{proof}
Let $\{v_{\{x,y,z\}}\,|\,x,y,z\mbox{ distinct elements of }\{1,\ldots,n\}\}$ be the standard basis of $M_3$. Define $\psi\in\Hom_{A_n}(M_3\da_{A_n},\End_F(V))$ through
\[\psi(v_{\{x,y,z\}})(w)=(x,y,z)w+(x,z,y)w\]
for each $w\in V$ (it can be easily checked that $\psi$ is an homomorphism). Let
\[e:=v_{\{1,2,3\}}+v_{\{1,5,6\}}+v_{\{4,2,6\}}+v_{\{4,5,3\}}-v_{\{1,2,6\}}-v_{\{1,5,3\}}-v_{\{4,2,3\}}-v_{\{4,5,6\}}.\]
Then $e$ generates $S_3$ (see\cite[Section 8]{JamesBook}), since it corresponds to the tableau
\[\begin{array}{ccccc}
4&5&6&\cdots&n.\\
1&2&3
\end{array}\]
Notice that $\psi(e)(w)=x_3 w$. 
Similar to \cite[Lemma 6.1]{kmt}, $\psi$ vanishes on $S_3\da_{A_n}$ if and only if $x_3 E^\la_\pm=0$.
\end{proof}

\begin{lemma}\label{l16}{\cite[Lemma 6.1]{kmt}}
Let $n\geq 8$ and $V$ be a $F\s_n$-module. For pairwise distinct $a,b,c,d$ define $[a,b,c,d]$ to be the sum of all elements of $\s_{\{a,b,c,d\}}$ which do not fix any element. If
\begin{align*}
x_4&=[1,2,3,4]+[5,6,3,4]+[5,2,7,4]+[5,2,3,8]+[1,6,7,4]+[1,6,3,8]\\
&\hspace{12pt}+[1,2,7,8]+[5,6,7,8]-[5,2,3,4]-[1,6,3,4]-[1,2,7,4]-[1,2,3,8]\\
&\hspace{12pt}-[5,6,7,4]-[5,6,3,8]-[5,2,7,8]-[1,6,7,8]
\end{align*}
and $x_4V\not=0$ then there exists $\psi\in\Hom_{\s_n}(M_4,\End_F(V))$ which does not vanish on $S_4$.
\end{lemma}

\begin{lemma}{\cite[Lemma 6.1]{m2}}\label{l12}
Let $p\geq 3$, $n\geq 6$ and $V$ be a $F\s_n$-module. If
\begin{align*}
x_{2^2}&=(2,5)(3,6)-(3,5)(2,6)-(1,5)(3,6)+(1,6)(3,5)-(2,5)(1,6)\\
&\hspace{12pt}+(1,5)(2,6)-(2,4)(3,6)+(3,4)(2,6)+(1,4)(3,6)-(1,6)(3,4)\\
&\hspace{12pt}+(2,4)(1,6)-(1,4)(2,6)-(2,5)(3,4)+(3,5)(2,4)+(1,5)(3,4)\\
&\hspace{12pt}-(1,4)(3,5)+(2,5)(1,4)-(1,5)(2,4)
\end{align*}
and $x_{2^2}V\not=0$ then there exists $\psi\in\Hom_{\s_n}(M_{2^2},\End_F(V))$ which does not vanish on $S_{2^2}$.
\end{lemma}

\section{Homomorphism rings}\label{s1}

With the help of the two previous sections we will now show that in many cases there exist homomorphisms $M_\mu\to\End_F(D)$ which do not vanish on $S_\mu$. Existence of such homomorphisms will then be used to prove that often $V\otimes W$ is not irreducible. In the next lemma remember that $\be_n=(\lceil (n+1)/2\rceil,\lfloor(n-1)/2\rfloor)$ is the partition labeling the basic spin modules in characteristic 2.

\begin{lemma}{\cite[Corollary 6.4]{kmt}}\label{L2}
Let $p=2$ and $n\geq 5$. If $\la\in\Par_2(n)$ with $\la\not=(n),\be_n$, then there exists $\psi\in\Hom_{\s_n}(M_2,\End_F(D^\la))$ which does not vanish on $S_2$.
\end{lemma}

\begin{lemma}{\cite[Corollary 6.10]{kmt}}\label{L3}
Let $p=2$ and $n\geq 6$. If $\la\in\Par_2(n)$ with $h(\la)\geq 3$, then there exists $\psi\in\Hom_{\s_n}(M_3,\End_F(D^\la))$ which does not vanish on $S_3$.
\end{lemma}

\begin{lemma}\label{L1}
Let $p=2$ and $n\geq 7$. If $\la\in\Par_2(n)$ with $h(\la)\geq 3$ and $\la\in\Parinv_2(n)$, then there exists $\psi\in\Hom_{A_n}(M_3\da_{A_n}\,\End_F(E^\la_\pm))$ which does not vanish on $S_3\da_{A_n}$.
\end{lemma}

\begin{proof}
From \cite[Lemma 3.17]{kmt} and Lemma \ref{split2} we have that $E^{(4,2,1)}$ is a composition factor of $E^\la_\pm\da_{A_7}$. From Lemma \ref{l18} it is enough to prove that $x_3 E^\la_\pm\not=0$ (where $x_3$ is as in \cite[\S6.1]{kmt} or Lemma \ref{l18}), which follows from $x_3E^{(4,2,1)}\cong x_3 D^{(4,2,1)}\not=0$ by \cite[Lemma 6.9]{kmt}.
\end{proof}

\begin{lemma}\label{L9}
Let $p=2$ and $n\geq 10$. If $\la\in\Par_2(n)$ with $h(\la)\geq 4$, then there exists $\psi\in\Hom_{\s_n}(M_4,\End_F(D^\la))$ which does not vanish on $S_4$.
\end{lemma}

\begin{proof}
By Lemma \ref{L5}, $D^{(4,3,2,1)}$ is a composition factor of $D^\la\da_{\s_{10}}$. Since $(4,3,2,1)$ is a 2-core we have that $D^{(4,3,2,1)}\cong S^{(4,3,2,1)}$. From Lemma \ref{l16} it is enough to prove that $x_4D^\la\not=0$, where $x_4$ is as in Lemma \ref{l16}. 
In particular it is enough to prove that $x_4S^{(4,3,2,1)}\not=0$. If $v_t$ and $e_t$ are the standard basis elements of $M^{(4,3,2,1)}$ and $S^{(4,3,2,1)}$ respectively (see \cite[Section 8]{JamesBook}) it can be easily checked that $x_4e_s$ has non-zero coefficient for $v_y$, where
\[s=\begin{array}{cccc}
1&2&3&4\\
5&6&7\\
8&9\\
10
\end{array}\quad\mbox{and}\quad y=\begin{array}{cccc}
1&2&7&9\\
4&6&8\\
3&10\\
5
\end{array}
\]
and so the lemma follows.
\end{proof}

\begin{lemma}{\cite[Lemma 3.8]{ks}}\label{psi2}
Let $p=3$ and $n\geq 4$. If $\la\in\Par_3(n)$ with $\la\not=(n),(n)^\Mull$, then there exists $\psi\in\Hom_{\s_n}(M_2,\End_F(D^\la))$ which does not vanish on $S_2$.
\end{lemma}

\begin{lemma}{\cite[Corollary 6.7]{kmt}}\label{psi3}
Let $p=3$ and $n\geq 6$. If $\la\in\Par_3(n)$ with $h(\la),h(\la^\Mull)\geq 3$, then there exists $\psi\in\Hom_{\s_n}(M_3,\End_F(D^\la))$ which does not vanish on $S_3$.
\end{lemma}

\begin{lemma}\label{l33}
Let $p=3$ and $n\geq 8$. If $\la\in\Par_3(n)$ with $h(\la),h(\la^\Mull)\geq 4$, then there exists $\psi\in\Hom_{\s_n}(M_4,\End_F(D^\la))$ which does not vanish on $S_4$.
\end{lemma}

\begin{proof}
By Lemma \ref{l16} in order to prove the lemma it is enough to prove that $x_4D^\la\not=0$ (where $x_4$ is as in Lemma \ref{l16}). Using Lemma \ref{l34} it is enough to prove the lemma when $n=8$. So we may assume that $\la=(4,2,1,1)$. Since $(4,2,1,1)$ is a 3-core, $D^{(4,2,1,1)}\cong S^{(4,2,1,1)}$. Let
\[\{v_{\{i,j\},k,l}|i,j,k,l\mbox{ distinct elements of }\{1,\ldots,8\}\}\]
be the standard basis of $M^{(4,2,1,1)}$. Let $e$ be the basis element of $S^{(4,2,1,1)}$ corresponding to the tableau
\[\begin{array}{cccc}
1&5&7&8\\
2&6&&\\
3&&&\\
4&&&
\end{array}\]
(see \cite[Section 8]{JamesBook} for definition of $e$). Then it can be proved that the coefficient of $x_4 e$ corresponding to $v_{\{2,3\},1,8}$ is non-zero and so the lemma holds.
\end{proof}

\begin{lemma}\label{l9}
Let $p=3$, $n\geq 6$ and $\la=(n-k,k)$ with $n-2k\geq 2$ and $k\geq 2$. Then there exists $\psi\in\Hom_{\s_n}(M_{2^2},\End_F(D^\la))$ which does not vanish on $S_{2^2}$.
\end{lemma}

\begin{proof}
Similar to the previous lemmas, from Lemmas \ref{l13} and \ref{l12} it is enough to prove that $x_{2^2}D^{(4,2)}\not=0$ (with $x_{2^2}$ as in Lemma \ref{l12}). Notice that $D^{(4,2)}\cong S^{(4,2)}$. Let $\{v_{\{i,j\}}:1\leq i<j\leq 6\}$ be the standard basis of $M^{(4,2)}$ and $e$ be the basis element of $S^{(4,2)}$ corresponding to the tableau
\[\begin{array}{cccc}
1&3&5&6\\
2&4&&
\end{array}\]
(see \cite[Section 8]{JamesBook} for definition of $e$). It can be computed that the coefficient of $x_{2^2}e$ corresponding to $v_{\{1,5\}}$ is non-zero, thus proving the lemma.
\end{proof}

\begin{lemma}\label{l30}
Let $p=3$ and $n\geq 9$. If $\la\in\Parinv_3(n)$ then there exists $\psi\in\Hom_{A_n}(M_3\da_{A_n},\End_F(E^\la_\pm))$ which does not vanish on $S_3\da_{A_n}$.
\end{lemma}

\begin{proof}
From Lemma \ref{Mull} we have that $h(\la)\geq 3$. Note that there are no Mullineux fixed partitions for $p=3$ and $n=9$. 
In view of \cite[Lemma 3.16]{kmt} there exists $\mu\in\Par_3(9)$ with $h(\mu),h(\mu^\Mull)\geq 3$ and $E^\mu$ a composition factor of $E^\la_\pm\da_{A_9}$. By \cite[Lemma 6.6]{kmt} we have that $x_3E^\mu\cong x_3D^\mu\not=0$. In particular $x_3E^\la_\pm\not=0$ and so the lemma holds by Lemma \ref{l18}.
\end{proof}

\begin{lemma}\label{l8}
Let $p=3$, $n> 6$ and $\la\in\Parinv_3(n)$ be a JS-partition with $h(\la)=3$. Then there exists $\psi\in\Hom_{\s_n}(M_{2^2},\End_F(D^\la))$ which does not vanish on $S_{2^2}$.
\end{lemma}

\begin{proof}
From Lemmas \ref{l10} and \ref{l12} it is enough to prove that $x_{2^2}D^{(5,1^2)}\not=0$ (with $x_{2^2}$ as in Lemma \ref{l12}). Notice that $D^{(5,1^2)}\cong S^{(5,1^2)}$ (see \cite[Tables]{JamesBook}). Let $\{v_{i,j}:i\not=j\in\{1,\ldots 7\}\}$ be the standard basis of $M^{(5,1^2)}$ and $\{e_{i,j}:2\leq i<j\leq 7\}$ be the standard basis of $S^{(5,1^2)}$ (see \cite[Section 8]{JamesBook}). It can be checked that the coefficient of $x_{2^2}e_{2,4}$ corresponding to $v_{2,5}$ is non-zero and so the lemma follows.
\end{proof}

\section{Permutation modules}\label{spm}

In this section we consider the structure of certain permutation modules $M^\alpha$. The structure of many of the modules considered here has already been studied in other papers, in some cases dual filtrations to those presented here where found. Note that if $M\sim N_1|\ldots|N_h$ then $M^*\sim N_h^*|\ldots|N_1^*$. As noted in section \ref{snot}, the modules $M^\la$, $Y^\la$ and $D^\la$ are self-dual. Remember that $M_\mu:=M^{(n-m,\mu)}$ and similarly for $S_\mu$, $D_\mu$ and $Y_\mu$ if $\mu\in\Par(m)$.

\begin{lemma}\label{L12e}{\cite[Lemmas 4.7 and 4.9]{kmt}}
Let $p=2$. If $n\geq 6$ is even then $M_1\cong D_0|D_1|D_0\sim S_1|D_0$ and $M_2\sim S_2|(D_0\oplus S_1)$. Further if $n\equiv 0\Md 4$ then
\[M_3\cong M_1\oplus(\overbrace{D_2|D_1|D_3}^{S_3}|\overbrace{D_1|D_2}^{S_2}).\]
\end{lemma}

\begin{lemma}\label{L12o}
Let $p=2$. If $n\geq 7$ is odd then
\[M_1\cong D_0\oplus D_1,\quad M_2\sim S_2|M_1,\quad M_3\sim S_3|M_2.\]
\end{lemma}

\begin{proof}
This follows from \cite[Lemma 4.6]{kmt}, since $\hd(S_k)\cong D_k$ for $0\leq k<n/2$ (in particular in these cases $S_k$ is indecomposable) and $S_k\subseteq M_k$.
\end{proof}

\begin{lemma} \label{L160817_2} 
Let $p=3$, $n\geq 8$ with $n\equiv 2\pmod{3}$. Then 
$$M_1\cong D_0\oplus D_1,\quad M_2\sim S_2|M_1,\quad M_3\sim S_3|M_2.$$
\end{lemma}

\begin{proof}
This holds by \cite[Lemma 4.5]{kmt}, since $\hd(S_k)\cong D_k$ for $0\leq k\leq n/2$ and $S_k\subseteq M_k$.
\end{proof}

\begin{lemma} \label{L160817_0}
Let $p=3$, $n\equiv 0\pmod{3}$ with $n\geq 9$. Then
\begin{align*}
&M_1\cong \overbrace{D_0|D_1}^{S_1}|D_0,&&M_2\cong D_2\oplus M_1,\\
&M_3\sim D_2\oplus (S_3|(D_0\oplus S_1)),&&M_4\sim S_4|S_1|A,\\
&M_{1^2}\sim M_2\oplus (S_{1^2}|S_1).
\end{align*}
for a module $A\subseteq M_3$ with $M_3/A\cong S_1$.
\end{lemma}

\begin{proof}
For the structure of $M_1$, $M_2$ and $M_3$ see \cite[Lemma 4.3]{kmt}, together from $S_k\subseteq M_k$ and $\hd(S_k)\cong D_k$ for $k\leq n/2$. It then also follows that $S_2\cong D_2$.

For the structure of $M_{1^2}$ note that by Lemma \ref{L160817_2} and \cite[Corollary 17.14]{JamesBook}
\[M_{1^2}
\cong M_1\oplus (D^{(n-2,1)}\ua^{\s_n})\cong M_1\oplus (S^{(n-2,1)}\ua^{\s_n})\sim M_1\oplus S_2\oplus (S_{1,1}|S_1).\]
We then only still have to study the structure of $M_4$.

For $0\leq k\leq n/2$ let
\[\{v_I:I\subseteq \{1,\ldots,n\}\text{ with }|I|=k\}\]
be the standard basis of $M_k$. Given $0\leq k\leq\ell\leq n/2$ define $\eta_{\ell,k}:M_\ell\to M_k$ by
\[\eta_{\ell,k}v_I=\sum_{{J\subseteq I\subseteq\{1,\ldots,n\}:}\atop{|J|=k}}v_J\]
for any element $v_I$ of the standard basis of $M_\ell$.

From \cite[Theorem 1]{Wil} we have that $\dim\Im\eta_{4,3}=\dim M_3-(n-1)$, $\dim\Im\eta_{4,1}=\dim M_1$ and $\dim\Im\eta_{3,1}=n-1$. In particular there exist submodules $X,Y\subseteq M_4$ and $A\subseteq M_3$ with $\dim A=\dim M_3-(n-1)$ such that $M_4\sim X|A$ and $M_4\sim Y|M_1$. Further $\eta_{3,1}\circ\eta_{4,3}=0$ by \cite[(3.1)]{Wil}. So $A\cong \ker\eta_{3,1}$. So $M_3/A\cong\Im\eta_{3,1}\subseteq M_1$ has dimension $n-1$. Since $M_1\cong D_0|D_1|D_0\sim S_1|D_0$ is uniserial and $D_0\cong \1_{\s_n}$, it then follows that $M_3/A\cong S_1$. Since $D_3\cong\hd(S_3)$ is not a composition factor of $S_1$ and $S_3\subseteq M_3$, it follows that $S_3\subseteq A$. From \cite[Example 17.17, Theorem 24.15]{JamesBook} we also have that $D_1\cong \hd(S_1)$ is not a composition factor of $A/S_3$. Since $D_0$ is contained exactly once in the head of $M_k$ for each $k$, it follows that $M_4\sim (X\cap Y)|S_1|A$. As $S_4\subseteq M_4$ and $D_4\cong \hd(S_4)$ is not a composition factor of $S_1$ or $A$, it follows by comparing dimensions that $M_4\sim S_4|S_1|A$.
\end{proof}

\begin{lemma} \label{L160817_1}
Let $p=3$, $n\equiv 1\pmod{3}$ with $n\geq 10$. Then
\begin{align*}
&M_1\cong D_0\oplus D_1,&&M_2\cong D_1\oplus (\overbrace{D_0|D_2}^{S_2}|D_0),\\
&M_3\sim D_1\oplus (S_3|(D_0\oplus S_2)),&&M_4\sim S_4|M_3.
\end{align*}
\end{lemma}

\begin{proof}
For the structure of $M_1$, $M_2$ and $M_3$ see \cite[Lemma 4.4]{kmt} and use that $\hd(S_k)\cong D_k$ for $k\leq n/2$. Notice that $D_1$ and $D_4$ are in the same block, while $D_0$, $D_2$ and $D_3$ are in a different block. Further $S_4\cong D_4$ or $S_4\cong D_1|D_4$ from \cite[Theorem 24.15]{JamesBook}. In particular 
$Y_4\cong D_4$ if $S_4\cong D_4$ or $Y_4\cong D_1|D_4|D_1$ if $S_4\cong D_1|D_4$. The lemma then follows from Lemma \ref{LYoung} and by comparing composition factors (see \cite[Example 17.17, Theorem 24.15]{JamesBook}).
\end{proof}

\section{Partitions with at least 2 normal nodes}\label{geq2n}

In the next three sections we will study more in details the endomorphism rings of the modules $D^\la$, $E^\la$ or $E^\la_\pm$ for certain particular classes of partitions. We start here by considering the case where $\la$ has at least 2 normal nodes.

\begin{lemma}\label{L7}
Let $p=2$ and $n\geq 10$ be even. If $\la\in\Par_2(n)$ with $\epsilon_0(\la)+\epsilon_1(\la)\geq 3$ then there exist $\psi,\psi',\psi''\in\Hom_{\s_n}(M_2,\End_F(D^\la))$ such that $\psi|_{S_2}$, $\psi'|_{S_2}$ and $\psi''|_{S_2}$ are linearly independent.
\end{lemma}

\begin{proof}
By Lemma \ref{L12e} we have that $M_2\sim S_2|(D_0\oplus S_1)$. By Lemma \ref{l2} if
\[\dim\End_{\s_{n-2,2}}(D^\la\da_{\s_{n-2,2}})=\dim\Hom_{\s_n}(S_1,\End_F(D^\la))+c,\]
then there exist homomorphisms $\phi_i\in\Hom_{\s_n}(M_2,\End_F(D^\la))$ for $1\leq i\leq c-1$ such that $\phi_1|_{S_2},\ldots,\phi_{c-1}|_{S_2}$ are linearly independent. Since $\la$ has at least 3 normal nodes we have by \cite[Lemma 5.4]{kmt} and Lemmas \ref{l2} and \ref{l12} that
\[\dim\End_{\s_{n-2}}(D^\la\da_{\s_{n-2}})>2\dim\Hom_{\s_n}(S_1,\End_F(D^\la))+7\]
and so by \cite[Lemma 4.14]{m1}
\[\dim\End_{\s_{n-2,2}}(D^\la\da_{\s_{n-2,2}})\geq \dim\Hom_{\s_n}(S_1,\End_F(D^\la))+4,\]
from which the lemma follows.
\end{proof}

\begin{lemma}\label{L8}
Let $p=2$ and $n\geq 10$ be even. Assume that $\la\in\Parinv_2(n)$ with $\epsilon_0(\la)+\epsilon_1(\la)=2$. Then there exist $\psi,\psi'\in\Hom_{\s_n}(M_2,\End_F(D^\la))$ such that $\psi|_{S_2}$ and $\psi'|_{S_2}$ are linearly independent.
\end{lemma}

\begin{proof}
By Lemma \ref{L12e} we have that $M_2\sim S_2|(D_0\oplus S_1)$ and so by Lemma \ref{l2} it is enough to prove that
\[\dim\End_{\s_{n-2,2}}(D^\la\da_{\s_{n-2,2}})\geq \dim\Hom_{\s_n}(S_1,\End_F(D^\la))+3.\]
From \cite[Lemma 5.5]{kmt} we have that $\dim\End_{\s_{n-2,2}}(D^\la\da_{\s_{n-2,2}})\geq 4$. By \cite[Lemmas 3.12, 3.13]{kmt} we may then assume that $\dim\Hom_{\s_n}(S_1,\End_F(D^\la))=2$ and that for some residue $\ell$ we have $\eps_\ell(\la),\phi_\ell(\la)>0$ and $(\la\setminus X)\cup Y$ is not $p$-regular, where $X$ is the $\ell$-good node and $Y$ the $\ell$-cogood node of $\la$. By \cite[Lemma 2.13]{kmt} we then have that $h(\la)\geq 3$ and that there exists $1\leq j\leq h(\la)$ with $\la_j=\la_{j+1}+2$ and
\[\la_1\equiv\ldots\equiv\la_{j-1}\not\equiv\la_j\equiv\la_{j+1}\not\equiv\la_{j+2}\equiv\ldots\equiv\la_{h(\la)}\Md 2.\]
If $j$ is odd then there exists $k\geq 1$ such that $\la_{2k+1}\geq 1$ and
\[\la_1\equiv \la_2\not\equiv\la_{2k+1}\equiv\la_{2k+2}\Md 2.\]
From Lemma \ref{split2} this contradicts the assumption that $D^\la\da_{A_n}$ splits. So $j$ is even. If $j=h(\la)$ then $\la_{h(\la)}=2$ and the other parts of $\la$ are odd, contradicting $n$ being even. If $j=h(\la)-1$ then, from Lemma \ref{split2}, $\la_{h(\la)-1}=3$, $\la_{h(\la)}=1$ and the other parts of $\la$ are even. So again from Lemma \ref{split2}, $\la=(4,3,1)$, contradicting $n\geq 10$. Thus $2\leq j\leq h(\la)-2$ is even. Notice that the normal nodes of $\la$ are on rows 1 and $j$ and so they have the same residue $i$. It then follows from Lemmas \ref{l45} and \ref{l39} that $D^\la\da_{\s_{n-2,2}}\cong A\oplus B$ with $A\da_{\s_{n-2}}\cong e_i^2D^\la$ and $B\da_{\s_{n-2}}\cong e_{1-i}e_iD^\la$. From \cite[Lemma 4.15]{m1} we have that $A\cong (D^{\tilde{e}^2_i(\la)}\otimes D^{(2)})|(D^{\tilde{e}^2_i(\la)}\otimes D^{(2)})$. So it is enough to prove that
\[\dim\End_{\s_{n-2,2}}(B)\geq\dim\Hom_{\s_n}(S_1,\End_F(D^\la))-\dim\End_{\s_{n-2,2}}(A)+3=3.\]
Notice that $B$ is self-dual, since it is a block component of a self-dual module of $\s_{n-2,2}$. Further
\[\tilde{e}_i(\la)=(\la_1,\ldots,\la_{j-1},\la_j-1,\la_{j+1},\ldots,\la_{h(\la)})\]
and then from $2\leq j\leq h(\la)-2$,
\[\tilde{e}_i(\la)_1\equiv\ldots\equiv\tilde{e}_i(\la)_j\not\equiv\tilde{e}_i(\la)_{j+1}\not\equiv\tilde{e}_i(\la)_{j+2}\equiv\ldots\equiv\tilde{e}_i(\la)_{h(\la)}\Md 2.\]
So $\epsilon_{1-i}(\tilde{e}_i(\la))=2$ (the corresponding normal nodes are on rows $j+1$ and $j+2$). Let $\mu:=\tilde{e}_{1-i}\tilde{e}_i(\la)$. From Lemma \ref{l39} it follows that
\begin{align*}
e_{1-i}e_iD^\la&\sim e_{1-i}D^{\tilde{e}_i(\la)}|\ldots|e_{1-i}D^{\tilde{e}_i(\la)}\sim\overbrace{D^{\mu}|\ldots|D^{\mu}}^{C}|\ldots|\overbrace{D^{\mu}|\ldots|D^{\mu}}^{C},
\end{align*}
with $C=e_{i-1}D^{\tilde{e}_i(\la)}$ indecomposable with simple head and socle and $[C:D^{\mu}]=2$. So $B$ is not semisimple. If the socle of $B$ is not simple then $\dim\End_{\s_{n-2,2}}(B)\geq 3$ (since head and socle of $B$ are isomorphic and $B$ is not semisimple). So we may assume that the socle of $B$ is simple. Since
\[\dim\Hom_{\s_{n-2,2}}(D^{\mu}\otimes M^{(1^2)},B)=\dim\Hom_{\s_{n-2}}(D^{\mu},B\da_{\s_{n-2}})\geq 1\]
and any composition factor of $M^{(1^2)}$ is of the form $D^{(2)}$, we then have that $\soc(B)\cong D^{\tilde{e}_{1-i}\tilde{e}_i(\la)}\otimes D^{(2)}$. Further
\begin{align*}
\dim\Hom_{\s_{n-2,2}}(C\otimes M^{(1^2)},B)&=\dim\Hom_{\s_{n-2}}(C,B\da_{\s_{n-2}})\\
&>\dim\Hom_{\s_{n-2}}(D^{\mu},B\da_{\s_{n-2}})\\
&=\dim\Hom_{\s_{n-2,2}}(D^{\mu}\otimes M^{(1^2)},B)\\
&\geq 1.
\end{align*}
Note that
\[\soc(B)\cong D^{\mu}\otimes D^{(2)}\cong \hd(C\otimes M^{(1^2)}).\]
So there exists a quotient $\overline{C}$ of $C\otimes M^{(1^2)}$ not isomorphic to $D^{\mu}\otimes D^{(2)}$ such that $\overline{C}\subseteq B$. Further $\soc(B)\subsetneq \overline{C}\subseteq B$ and $\overline{C}$ has simple head and socle isomorphic to $D^{\mu}\otimes D^{(2)}$. If $\overline{C}\cong C\otimes M^{(1^2)}$ then $\overline{C}$ is self-dual, as is $B$. So $C\otimes M^{(1^2)}$ is also a quotient of $B$ and then
\[\dim\End_{\s_{n-2,2}}(B)\geq\dim\End_{\s_{n-2,2}}(C\otimes M^{(1^2)})=4\]
(using Lemma \ref{l39}). So we may assume that $\overline{C}\not\cong C\otimes M^{(1^2)}$. Notice that $[C\otimes M^{(1^2)}:D^{\mu}\otimes D^{(2)}]=4$, that $C\otimes M^{(1^2)}$ has simple head and socle isomorphic to $D^{\mu}\otimes D^{(2)}$ and that $C\otimes D^{(2)}$ and $D^{\mu}\otimes M^{(1^2)}$ are distinct submodules of $C\otimes M^{(1^2)}$ with $[C\otimes D^{(2)}:D^{\mu}\otimes D^{(2)}]=2$, $[D^{\mu}\otimes M^{(1^2)}:D^{\mu}\otimes D^{(2)}]=2$ and both $C\otimes D^{(2)}$ and $D^{\mu}\otimes M^{(1^2)}$ have simple head and socle isomorphic to $D^{\mu}\otimes D^{(2)}$. So $[\overline{C}:D^{\mu}\otimes D^{(2)}]=2$.

Note that when $\mu\in\Par_p(m)$ and $D^\mu$ is defined as $K\s_m$-module (with $K$ a field which is not necessarily algebraic closed), then any block component of the restriction of $D^\mu$ to a Young subgroup is self-dual. Further any permutation module of $KG$ is self-dual, for any field $K$ and group $G$. In particular the previous part, about the structure of $B$, also holds over $\mathbb{F}_2$ (since $\mathbb{F}_2$ is a splitting field of $\s_m$ for any $m$) 
and not only over $F$, where $F$ is algebraically closed, so until the end of the proof we will work over the field $\mathbb{F}_2$.

In this case there exist exactly three submodules $E_1,E_2,E_3\subseteq C\otimes M^{(1^2)}$ with $[E_j:D^{\mu}\otimes D^{(2)}]=2$ and head and socle isomorphic to $D^{\mu}\otimes D^{(2)}$. Similarly $C\otimes M^{(1^2)}$ has exactly three quotients $F_1,F_2,F_3$ with $[F_j:D^{\mu}\otimes D^{(2)}]=2$ and head and socle isomorphic to $D^{\mu}\otimes D^{(2)}$. We may assume that $E_1\cong F_1\cong C\otimes D^{(2)}$ and that $E_2\cong F_2\cong D^{\mu}\otimes M^{(1^2)}$. Let $g_1,g_2\in\End_{\s_{n-2,2}}(C\otimes M^{(1^2)})$ with $\Im\, g_j=E_j$ and $(C\otimes M^{(1^2)})/\Ker\, g_j=F_j$. Since $C\otimes M^{(1^2)}$ has simple head and socle isomorphic to $D^{\mu}\otimes D^{(2)}$, then so does $\Im(g_1+g_2)$, if it is non-zero. Since $E_j\not\subseteq E_k$ if $j\not=k$, we then have that $E_3=\Im (g_1+g_2)$ and $(C\otimes M^{(1^2)})/\Ker (g_1+g_2)=F_3$. So $E_3\cong F_3$. By duality of $C\otimes M^{(1^2)}$, there exists $\sigma\in\s_3$ with $F_{\sigma(j)}^*\cong E_j$ for $1\leq j\leq 3$. Since $E_1\cong F_1$ and $E_2\cong F_2$ are self-dual, it then follows that also $E_3\cong F_3$ is self-dual. In particular $\overline{C}$ is self dual, since it is isomorphic to some $E_j$.

Since $\soc(B)\subsetneq \overline{C}\subsetneq B$ and any of these three modules is self-dual, it then follows that $\dim\End_{\s_{n-2,2}}(B)\geq 3$.
\end{proof}

\section{Two rows partitions}

Modules indexed by two rows partitions will play a special role in the proof of Theorem \ref{mt}, since in this case not all results from Section \ref{s1} apply. So we will consider them more in details in this section. We start by citing a branching result for two rows partitions, which is part of the main result of \cite{sh}, that will be used in this section.

We want to remember that when writing for example $D_1\subseteq\End_F(V)$ we mean that $\End_F(V)$ has a submodule which is isomorphic to $D_1$.

\begin{lemma}\label{tsh}
Let $\la=(n-k,k)$ with $k\geq 1$ and $n-2k\geq 1$. Write $n-2k=\sum_j s_jp^j$ with $0\leq s_j<p$ and let $t$ minimal such that $s_t<p-1$. If $t\geq 1$ then, in the Grothendieck group, $[D^\la\da_{\s_{n-1}}]$ is equal to
\[[D^{(n-k-1,k)}]+\de[D^{(n-k-1+p^t,k-p^t)}]+\sum_{j=0}^{t-1}2[D^{(n-k-1+p^j,k-p^j)}],\]
where $D^{(n-k-1+r,k-r)}:=0$ if $(n-k-1+r,k-r)\not\in\Par_p(n-1)$ and $\de=1$ if $s_t<p-2$ or $\de=0$ else.
\end{lemma}

\begin{lemma}\label{L4}
Let $p=2$ and $n\geq 7$ be odd. If $\la=(n-k,k)$ with $k\geq 2$ and $n-2k\geq 3$ then there exist $\psi_2,\psi_2'\in\Hom_{\s_n}(M_2,\End_F(D^\la))$ such that $\psi_2|_{S_2}$, $\psi_2'|_{S_2}$ are linearly independent or there exists $\psi_3\in\Hom_{\s_n}(M_3,\End_F(D^\la))$ which does not vanish on $S_3$.
\end{lemma}

\begin{proof}
From Lemma \ref{L12o}, $M_2\sim S_2|M_1$ and $M_3\sim S_3|M_2$. So if
\[\dim\End_{\s_{n-2,2}}(D^\la\da_{\s_{n-2,2}})\geq \dim\End_{\s_{n-1}}(D^\la\da_{\s_{n-1}})+2\]
there exist $\psi,\psi'\in\Hom_{\s_n}(M_2,\End_F(D^\la))$ such that $\psi|_{S_2}$, $\psi'|_{S_2}$ are linearly independent, by Lemma \ref{l2}. If
\[\dim\End_{\s_{n-3,3}}(D^\la\da_{\s_{n-3,3}})\geq \dim\End_{\s_{n-2,2}}(D^\la\da_{\s_{n-2,2}})+1\]
there exists $\psi\in\Hom_{\s_n}(M_3,\End_F(D^\la))$ which does not vanish on $S_3$, again by Lemma \ref{l2}.

Since $n$ is odd, both removable nodes are normal and so,  by Lemma \ref{l53}, $\dim\End_{\s_{n-1}}(D^\la\da_{\s_{n-1}})=2$. It is then enough to prove that at least one of
\[\dim\End_{\s_{n-2,2}}(D^\la\da_{\s_{n-2,2}})\geq 4\quad\text{or}\quad\dim\End_{\s_{n-3,3}}(D^\la\da_{\s_{n-3,3}})\geq 4\]
holds. Note that $n-2k$ is odd.

{\bf Case 1:} $n-2k\equiv 3\Md 4$, so $t\geq 2$ in Lemma \ref{tsh}. Then by block decomposition (Lemma \ref{l45}), Lemmas \ref{l39} and \ref{tsh}
\[D^\la\da_{\s_{n-2}}\cong (D^{(n-k-1,k-1)})^{\oplus 2}\oplus A\]
where $[A:D^{(n-k-2,k)}]=1$ and $[A:D^{(n-k,k-2)}]=2$. 
It easily follows that $D^\la\da_{\s_{n-2,2}}$ has (at least) 2 block components with at least 2 composition factors each and then $\dim\End_{\s_{n-2,2}}(D^\la\da_{\s_{n-2,2}})\geq 4$, since $F\s_n$- and $F\s_{n-2,2}$-modules are self-dual.

{\bf Case 2:} $n-2k\equiv 1\Md 4$, so $t=1$ in Lemma \ref{tsh}. 
Then by block decomposition (Lemma \ref{l45}),  Lemmas \ref{l39} and \ref{tsh}
\begin{align*}
D^\la\da_{\s_{n-1}}\cong &D^{(n-k,k-1)}|D^{(n-k-1,k)}|D^{(n-k,k-1)},\\
D^\la\da_{\s_{n-2}}\cong &(D^{(n-k-1,k-1)})^{\oplus 2}\oplus D^{(n-k-2,k)},\\
D^\la\da_{\s_{n-3}}\sim &(\overbrace{D^{(n-k-1,k-2)}|D^{(n-k-2,k-1)}|D^{(n-k-1,k-2)}}^B)^{\oplus 2}\\
&\oplus (\overbrace{D^{(n-k-2,k-1)}|\ldots|D^{(n-k-3,k)}|\ldots|D^{(n-k-2,k-1)}}^C),
\end{align*}
where $B
$ and $C
$ are indecomposable with simple head and socle. To see this, note that $D^\la\da_{\s_{n-1}}$ is indecomposable with simple head and socle each isomorphic to $D^{(n-k,k-1)}$ by Lemma \ref{l39} and the composition factors, with multiplicities, of $D^\la\da_{\s_{n-1}}$ are known by Lemma \ref{tsh}. The structure of $D^\la\da_{\s_{n-2}}$ then follows by Lemmas \ref{l45} and \ref{l39}. For $D^\la\da_{\s_{n-3}}$ use again Lemmas \ref{l39} and \ref{tsh}.

Notice that $D^\la\da_{\s_{n-3,3}}\cong F\oplus G$, where all composition factors of $F\da_{\s_{1^{n-3},3}}$ are of the form $\1\otimes D^{(3)}$ and all composition factors of $G\da_{\s_{1^{n-3},3}}$ are of the form $\1\otimes D^{(2,1)}$ (since $D^{(3)}$ and $D^{(2,1)}$ are in different blocks). From \cite[Lemma 1.11]{bk5} we have that $D^{(n-k-2,k-1)}\otimes D^{(2,1)}$ and $D^{(n-k-3,k)}\otimes D^{(3)}$ are composition factors of $D^\la\da_{\s_{n-3,3}}$. So $F$ has a composition factor isomorphic to $D^{(n-k-2,k-1)}\otimes D^{(2,1)}$. Since $D^{(2,1)}$ has dimension 2 and $D^{(n-k-2,k-1)}$ appears only once in the socle of $D^\la\da_{\s_{n-3}}$, it follows that $F$ is non-zero and not simple. Similarly $G$ is non-zero and not simple, since it has a composition factor $D^{(n-k-3,k)}\otimes D^{(3)}$ and $D^{(n-k-3,k)}$ does not appear in the socle of $D^\la\da_{\s_{n-3}}$. Further $F$ and $G$ are self-dual, since they are block components of $D^\la\da_{\s_{n-3,3}}$. So $\dim\End_{\s_{n-3,3}}(D^\la\da_{\s_{n-3,3}})\geq 4$.
\end{proof}

\begin{lemma}\label{Lemma7.1}
Let $p=2$ and $n\geq 4$ be even. If $\la=(n-k,k)$ with $1\leq k<n/2$ and $\dim\Hom_{\s_n}(S_1,\End_F(D^\la))\geq 1$ then $\la=\be_n$. In this case if $n\equiv 0\Md 4$ then $D_1\subseteq\End_F(D^\la)$, while if $n\equiv 2\Md 4$ then $S_1\subseteq\End_F(D^\la)$.
\end{lemma}

\begin{proof}
This follows from \cite[Lemma 7.1]{m1}, since $(n-k,k)$ is JS by Lemma \ref{L151119}.
\end{proof}

\begin{lemma}\label{L10}
Let $p=2$ and $n\geq 8$ with $n\equiv 0\Md 4$. If $\la=(n-k,k)$ with $k\geq 2$ and $n-2k\geq 3$ then one of the following happens:
\begin{itemize}
\item $D_2^{\oplus 2}\subseteq \End_F(D^\la)$,

\item $S_3\subseteq \End_F(D^\la)$,

\item $D_2\oplus D_3\subseteq\End_F(D^\la)$,

\item there exists $\psi\in\Hom_{\s_n}(M_4,\End_F(D^\la))$ which does not vanish on $S_4$.
\end{itemize}
\end{lemma}

\begin{proof}
Note that $\la$ is JS by Lemma \ref{L151119}, since $n$ is even and $\la$ has two parts. From Lemmas \ref{l39} and \ref{tsh} we have that
\begin{align*}
D^\la\da_{\s_{n-1}}&\cong D^{(n-k-1,k)},\\
D^\la\da_{\s_{n-2}}&\sim D^{(n-k-1,k-1)}|\overbrace{B|D^{(n-k-2,k)}|C}^N|D^{(n-k-1,k-1)},
\end{align*}
where all composition factors of $B$ and $C$ are of the form $D^{(n-k-2+2^i,k-2^i)}$ with $i\geq 1$. Let $0\leq j\leq \lfloor k/2\rfloor$ with $D^{(n-k-2+2j,k-2j)}\subseteq N$ (such a $j$ exists since any composition factor of $N$, and so also of its socle, is of the form $D^{(n-k-2+2\overline{j},k-2\overline{j})}$ for some $0\leq \overline{j}\leq \lfloor k/2\rfloor$). 
By Lemma \ref{l39} and block decomposition we then have that
\begin{align*}
D^\la\da_{\s_{n-3}}&\cong (D^{(n-k-2,k-1)})^{\oplus 2}\oplus N\da_{\s_{n-3}},\\
D^\la\da_{\s_{n-4}}&\supseteq D^{(n-k-2,k-2)}\oplus D^{(n-k-3+m,k-m-1)},
\end{align*}
where $m=2j$ if $j<k/2$ or $m=2j-1$ if $j=k/2$.

{\bf Case 1:} $k\geq 3$. Fix $j,m$ as above. We may assume that there is no $\psi\in\Hom_{\s_n}(M_4,\End_F(D^\la))$ which does not vanish on $S_4$. By \cite[Example 17.17]{JamesBook} we have that $M_4\sim S_4|A$ with $A\sim S_3|S_2|S_1|S_0$. Further $\dim\Hom_{\s_n}(S_1,\End_F(D^\la))=0$ by Lemma \ref{Lemma7.1}. In particular $D_1\not\subseteq\End_F(D^\la)$, since $D_1\cong\hd(S_1)$. By Lemma \ref{l2} it then follows that
\begin{align*}
&\dim\End_{\s_{n-4,4}}(D^\la\da_{\s_{n-4,4}})\\
&=\dim\Hom_{\s_n}(M_4,\End_F(D^\la))\\
&=\dim\Hom_{\s_n}(A,\End_F(D^\la))\\
&\leq\dim\Hom_{\s_n}(S_3,\End_F(D^\la))+\dim\Hom_{\s_n}(S_2,\End_F(D^\la))\\
&\hspace{11pt}+\dim\Hom_{\s_n}(S_1,\End_F(D^\la))+\dim\Hom_{\s_n}(S_0,\End_F(D^\la))\\
&=\dim\Hom_{\s_n}(S_3,\End_F(D^\la))+\dim\Hom_{\s_n}(S_2,\End_F(D^\la))+1.
\end{align*}
From Lemma \ref{L12e}, $S_3\cong D_2|D_1|D_3$ and $S_2\cong D_1|D_2$. From Lemma \ref{L2} we have that $D_2\subseteq \End_F(D^\la)$, since $D_1\not\subseteq\End_F(D^\la)$. By the same reasons, if $\dim\Hom_{\s_n}(S_2,\End_F(D^\la))\geq 2$ then $D_2^{\oplus 2}\subseteq\End_F(D^\la)$, while if $\dim\Hom_{\s_n}(S_3,\End_F(D^\la))\geq 1$ then $D_3$ or $S_3$ is contained in $\End_F(D^\la)$ (and so in this case $D_2\oplus D_3$ or $S_3$ is contained in $\End_F(D^\la)$. Thus it is enough to prove that $\dim\End_{\s_{n-4,4}}(D^\la\da_{\s_{n-4,4}})\geq 3$.

We have that $(n-k-2,k-2)\not=(n-k-3+m,k-m-1)$, since $k\geq 3$. If $\mu$ is either of these two partitions then
\[\dim\Hom_{\s_{n-4,4}}(D^\mu\otimes M^{(1^4)},D^\la\da_{\s_{n-4,4}})=\dim\Hom_{\s_{n-4}}(D^\mu,D^\la\da_{\s_{n-4}})\geq 1.\]
It then follows that there are at least two non-isomorphic simple modules appearing in the socle of $D^\la\da_{\s_{n-4,4}}$. Further $D^\la\da_{\s_{n-4}}$ is not semisimple, since it contains $D^{(n-k-2,k-1)}\da_{\s_{n-4}}$ which is not semisimple (by Lemma \ref{l39}). So the same holds for $D^\la\da_{\s_{n-4,4}}$. In particular $\dim\End_{\s_{n-4,4}}(D^\la\da_{\s_{n-4,4}})\geq 3$, since $D^\la\da_{\s_{n-4,4}}$ is self-dual, it is not semisimple and its socle is not simple.

{\bf Case 2:} $k=2$. By Lemma \ref{L12e} we have that $M_3\sim M_1\oplus (S_3|S_2)$. So by Lemmas \ref{l2} and \ref{l53}
\begin{align*}
&\dim\End_{\s_{n-3,3}}(D^\la\da_{\s_{n-3,3}})\\
&=\dim\Hom_{\s_n}(M_3,\End_F(D^\la))\\
&\leq\dim\Hom_{\s_n}(M_1,\End_F(D^\la))+\dim\Hom_{\s_n}(S_3,\End_F(D^\la))\\
&\hspace{11pt}+\dim\Hom_{\s_n}(S_2,\End_F(D^\la))\\
&=\dim\Hom_{\s_n}(S_3,\End_F(D^\la))+\dim\Hom_{\s_n}(S_2,\End_F(D^\la))+1.
\end{align*}
In this case it is then enough to prove that $\dim\End_{\s_{n-3,3}}(D^\la\da_{\s_{n-3,3}})\geq 3$. Since $n\equiv 0\Md 4$, from Lemma \ref{tsh} we have that $[N]=2[D^{(n-2)}]+[D^{(n-4,2)}]$ and so $[N\da_{\s_{n-3}}]=2[D^{(n-3)}]+[D^{(n-5,2)}]$. Since $D^{(n-3)}$ and $D^{(n-5,2)}$ are not in the same block as $D^{(n-4,1)}$, it follows that $D^\la\da_{\s_{n-3,3}}$ has at least two non-zero block components and that at least one of the block components is not simple. So $\dim\End_{\s_{n-3,3}}(D^\la\da_{\s_{n-3,3}})\geq 3$, since block components of $D^\la\da_{\s_{n-3,3}}$ are self-dual, as are simple $F\s_{n-3,3}$-modules.
\end{proof}

\begin{lemma}\label{l46}
Let $p=3$, $n\equiv 0\Md 3$ 
and $\la=(n-k,k)$ with $1\leq k<n/2$. Then 
\[\dim\Hom_{\s_n}(S_1,\End_F(D^\la))=\dim\End_{\s_{n-1}}(D^\la\da_{\s_{n-1}})-1.\]
\end{lemma}

\begin{proof}
From Lemma \ref{L160817_0} and self-duality of $M_1$ and $D_0$, we have that $M_1\sim D_0|S_1^*$. So
\[D^\la\otimes M_1\sim (D^\la\otimes D_0)|(D^\la\otimes S_1^*)\sim D^\la|(D^\la\otimes S_1^*)\]
and then there exists $D\subseteq D^\la\otimes M_1$ with $D\cong D^\la$ such that $D^\la\otimes S_1^*\cong (D^\la\otimes M_1)/D$. Since $p=3$ and $h(\la)=2$, from \cite[Theorem 2.10]{ks3} we have that $\Ext^1(D^\la,D^\la)=0$. So
\begin{align*}
\dim\Hom_{\s_n}(S_1,\End_F(D^\la))&=\dim\Hom_{\s_n}(D^\la,D^\la\otimes S_1^*)\\
&=\dim\Hom_{\s_n}(D^\la,D^\la\otimes M_1)-1\\
&=\dim\Hom_{\s_n}(M_1,\End_F(D^\la))-1\\
&=\dim\End_{\s_{n-1}}(D^\la\da_{\s_{n-1}})-1.
\end{align*}
\end{proof}

\begin{lemma}\label{l44}
Let $p=3$, $n\geq 10$ with $n\equiv 1\Md 3$ and $\la=(n-k,k)$ with $1\leq k<n/2$. Then
\begin{align*}
&\dim\End_{\s_{n-2,2}}(D^\la\da_{\s_{n-2,2}})-\dim\End_{\s_{n-1}}(D^\la\da_{\s_{n-1}})\\
&=\dim\Hom_{\s_n}(S_2,\End_F(D^\la)).
\end{align*}
\end{lemma}

\begin{proof}
From Lemma \ref{L160817_1} and self-duality of $M_2$, $M_1$ and $D_0$ we have that $M_2\oplus D_0\sim M_1\oplus (D_0|S_2^*)$. Similarly to the previous lemma we then have that
\begin{align*}
&\dim\Hom_{\s_n}(S_2,\End_F(D^\la))\\
&=\dim\Hom_{\s_n}(M_2\oplus D_0,\End_F(D^\la))-\dim\Hom_{\s_n}(M_1,\End_F(D^\la))-1\\
&=\dim\End_{\s_{n-2,2}}(D^\la\da_{\s_{n-2,2}})-\dim\End_{\s_{n-1}}(D^\la\da_{\s_{n-1}}).
\end{align*}
\end{proof}

\begin{lemma}\label{l41}
Let $p=3$, $n\geq 9$ and $\la=(n-k,k)$ with $1\leq k\leq n/2$. If
\[\dim\End_{\s_{n-3,3}}(D^\la\da_{\s_{n-3,3}})>\dim\End_{\s_{n-2,2}}(D^\la\da_{\s_{n-2,2}})\]
then there exists $\psi\in\Hom_{\s_3}(M_3,\End_F(D^\la))$ which does not vanish on $S_3$.
\end{lemma}

\begin{proof}
If $n\equiv 2\Md 3$ we have by Lemma \ref{L160817_2} that $M_3\sim S_3|M_2$ and so the lemma follows by Lemma \ref{l2} applied for both $\alpha=(n-2,2)$ and $(n-3,3)$.

If $n\equiv 0\Md 3$ then by Lemma \ref{L160817_0} we have that $M_3\sim D_2\oplus (S_3|(D_0\oplus S_1))$ and that $M_2\cong D_2\oplus M_1$. Again by Lemma \ref{l2} applied for both $\alpha=(n-2,2)$ and $(n-3,3)$ and by assumption
\begin{align*}
\dim\Hom_{\s_n}(M_3,\End_F(D^\la))&=\dim\End_{\s_{n-3,3}}(D^\la\da_{\s_{n-3,3}})\\
&>\dim\End_{\s_{n-2,2}}(D^\la\da_{\s_{n-2,2}})\\
&=\dim\Hom_{\s_n}(M_2,\End_F(D^\la)).
\end{align*}
Since $M_2\cong D_2\oplus M_1$ we then have
\[\dim\Hom_{\s_n}(M_3,\End_F(D^\la))>\dim\Hom_{\s_n}(D_2\oplus M_1,\End_F(D^\la))\]
and then from Lemma \ref{l46}
\begin{align*}
\dim\Hom_{\s_n}(M_3,\End_F(D^\la))&>\dim\Hom_{\s_n}(D_2\oplus S_1,\End_F(D^\la))+1\\
&=\dim\Hom_{\s_n}(D_2\oplus D_0\oplus S_1,\End_F(D^\la)).
\end{align*}
Since $M_3\sim S_3|(D_2\oplus D_0\oplus S_1)$, the lemma follows.

If $n\equiv 1\Md 3$ then by Lemma \ref{L160817_1} we have that $M_3\sim D_1\oplus (S_3|(D_0\oplus S_2))$ and $M_1\cong D_0\oplus D_1$. The result then follows by Lemma \ref{l44} similarly to the previous case.
\end{proof}

\begin{lemma}\label{l35}
Let $p=3$, $n\equiv 1\Md 3$ with $n\geq 10$ and $\la=(n-k,k)$ with $2\leq k<n/2$ and $n-2k\geq 2$. If \[\dim\End_{\s_{n-4,4}}(D^\la\da_{\s_{n-4,4}})>\dim\End_{\s_{n-3,3}}(D^\la\da_{\s_{n-3,3}})\]
then there exists $\psi\in\Hom_{\s_n}(M_4,\End_F(D^\la))$ which does not vanish on $S_4$.
\end{lemma}

\begin{proof}
From  Lemma \ref{L160817_1} we have that $M_4\sim S_4|M_3$. The result then follows by Lemma \ref{l2} applied for both $\alpha=(n-3,3)$ and $(n-4,4)$.
\end{proof}

\begin{lemma}\label{l43}
Let $p=3$ and $\la=(n-k,k)$ with $n-2k\geq 2$ and $k\geq 2$. If the two removable nodes of $\la$ are both normal and have different residues then
\[\dim\End_{\s_{n-2,2}}(D^\la\da_{\s_{n-2,2}})=3,\quad\dim\End_{\s_{n-3,3}}(D^\la\da_{\s_{n-3,3}})\geq 4.\]
\end{lemma}

\begin{proof}
In this case $n-k\equiv k\Md 3$, so $n-2k\geq 3$. Also if $i$ is the residue of the removable node on the first row of $\la$, then the residue of the removable node on the second row of $\la$ is $i-1$. Considering residues of removable/addable nodes of the corresponding partitions, it follows easily from Lemmas \ref{l45} and \ref{l39} that $e_i D^\la\cong D^{(n-k-1,k)}$, $e_{i-1} D^\la\cong D^{(n-k,k-1)}$ and
\begin{align*}
D^\la\da_{\s_{n-2}}&\cong D^{(n-k-1,k)}\da_{\s_{n-2}}\oplus D^{(n-k,k-1)}\da_{\s_{n-2}}\\
&\cong e_{i-1}D^{(n-k-1,k)}\oplus e_iD^{(n-k,k-1)}.
\end{align*}
Further $e_iD^{(n-k,k-1)}\cong D^{(n-k-1,k-1)}$, while $e_{i-1}D^{(n-k-1,k)}$ has simple socle isomorphic to $D^{(n-k-1,k-1)}$ and $\dim\End_{\s_{n-2}}(e_{i-1}D^{(n-k-1,k)})=2$.

Note that
\begin{align*}
D^\la\da_{\s_{n-2,2}}&\subseteq D^\la\da_{\s_{n-2}}\ua^{\s_{n-2,2}}\\
&\cong (D^{(n-k-1,k-1)}\oplus e_{i-1}D^{(n-k-1,k)})\otimes (D^{(2)}\oplus D^{(1^2)}).
\end{align*}
From \cite[Lemma 1.11]{bk5} we have that $D^{(n-k-2,k)}\otimes D^{(2)}$ and $D^{(n-k-1,k-1)}\otimes D^{(1^2)}$ are both composition factors of $D^\la\da_{\s_{n-2,2}}$. Since $\soc(D^\la\da_{\s_{n-2}})\cong (D^{(n-k-1,k-1)})^{\oplus 2}$, it follows (by block decomposition) that
\[\soc(D^\la\da_{\s_{n-2,2}})\cong D^{(n-k-1,k-1)}\otimes (D^{(2)}\oplus D^{(1^2)})\]
and that $D^\la\da_{\s_{n-2,2}}\cong M\oplus N$ with $M$ and $N$ indecomposable with simple socle. Since $D^{(n-k-1,k-1)}\subseteq e_{i-1}D^{(n-k-1,k)}$, by \cite[Lemma 1.2]{bk2} we have that, up to exchange, $M\subseteq e_{i-1}D^{(n-k-1,k)}\otimes D^{(2)}$ and $N\subseteq e_{i-1}D^{(n-k-1,k)}\otimes D^{(1^2)}$. By the same lemma we also have that $e_{i-1}D^{(n-k-1,k)}\subseteq M\da_{\s_{n-2}}$ or $N\da_{\s_{n-2}}$. Thus
\[D^\la\da_{\s_{n-2,2}}\cong (D^{(n-k-1,k-1)}\otimes D^{(2)})\oplus (e_{i-1}D^{(n-k-1,k)}\otimes D^{(1^2)})\]
or
\[D^\la\da_{\s_{n-2,2}}\cong (D^{(n-k-1,k-1)}\otimes D^{(1^2)})\oplus (e_{i-1}D^{(n-k-1,k)}\otimes D^{(2)}).\]
So
\[\dim\End_{\s_{n-2,2}}(D^\la\da_{\s_{n-2,2}})=1+\dim\End_{\s_{n-2}}(e_{i-1}D^{(n-k-1,k)})=3.\]

Since $k\geq 2$, from Lemma \ref{l39} we also have that $e_{i\pm1} D^{(n-k-1,k-1)}\not=0$. Since $D^{(n-k-1,k-1)}$ appears with multiplicity larger than 1 in $D^\la\da_{\s_{n-2}}$ and all simple $F\s_3$-modules are 1-dimensional, it follows that $D^\la\da_{\s_{n-3,3}}$ has at least two blocks components which are non-zero and not simple. As block components of $D^\la\da_{\s_{n-3,3}}$ are self-dual, we then have that
\[\dim\End_{\s_{n-3,3}}(D^\la\da_{\s_{n-3,3}})\geq 4.\]
\end{proof}

\begin{lemma}\label{l38}
Let $p=3$, $n\geq 9$ and $\la=(n-k,k)$ with $n-2k\geq 2$ and $k\geq 2$. If the two removable nodes of $\la$ are both normal and have different residues then there exists $\psi\in\Hom_{\s_n}(M_3,\End_F(D^\la))$ which does not vanish on $S_3$.
\end{lemma}

\begin{proof}
From Lemma \ref{l43} we have that
\[\dim\End_{\s_{n-3,3}}(D^\la\da_{\s_{n-3,3}})>\dim\End_{\s_{n-2,2}}(D^\la\da_{\s_{n-2,2}}).\]
The result then holds by Lemma \ref{l41}.
\end{proof}

\begin{lemma}\label{l42}
Let $p=3$ and $\la=(n-k,k)$ with $n-2k\geq 2$ and $k\geq 2$. If the two removable nodes of $\la$ have the same residue then
\[\dim\End_{\s_{n-1}}(D^\la\da_{\s_{n-1}})=2,\quad\dim\End_{\s_{n-2,2}}(D^\la\da_{\s_{n-2,2}})\geq 4.\]
\end{lemma}

\begin{proof}
In this case $n-k\equiv k+2\Md 3$ and both removable nodes are normal. It then follows that $\dim\End_{\s_{n-1}}(D^\la\da_{\s_{n-1}})=2$ by Lemma \ref{l53}. 
Let $i$ be the residue of the removable nodes of $\la$. From Lemma \ref{l39} and considering the structure of the corresponding partitions $[e_i^2 D^\la:D^{(n-k-1,k-1)}]=2$ and
\begin{align*}
[e_{i-1}e_i D^\la:D^{(n-k,k-2)}]&\geq [e_i D^\la:D^{(n-k,k-1)}]\cdot [e_{1-i}D^{(n-k,k-1)}:D^{(n-k,k-2)}]\\
&=2.
\end{align*}
So by block decomposition $D^\la\da_{\s_{n-2}}\cong A\oplus B$ with $A$ and $B$ non-zero, non-simple and self-dual. Since any simple $F\s_2$-module is 1-dimensional it is easy to see that a similar decomposition exists for $D^\la\da_{\s_{n-2,2}}$. The lemma then follows.
\end{proof}

\begin{lemma}\label{l37}
Let $p=3$ 
and $\la=(n-k,k)$ with $n-2k\geq 2$ and $k\geq 2$. If the two removable nodes of $\la$ have the same residue then there exists $\psi,\psi'\in\Hom_{\s_n}(M_2,\End_F(D^\la))$ such that $\psi|_{S_2}$ and $\psi'|_{S_2}$ are linearly independent.
\end{lemma}

\begin{proof}
From Lemma \ref{l42} we have that
\[\dim\End_{\s_{n-1}}(D^\la\da_{\s_{n-1}})\geq \dim\End_{\s_{n-2,2}}(D^\la\da_{\s_{n-2,2}})+2.\]
We have that $M_2\sim S_2|M_1$ by Lemmas \ref{L160817_2}, \ref{L160817_0} and \ref{L160817_1} (if $n\equiv 0\Md 3$ then $S_2\cong D_2$ by \cite[Theorem 24.15]{JamesBook}). The result then follows from Lemma \ref{l2}.
\end{proof}

\begin{lemma}\label{l36}
Let $p=3$ and $\la=(n-k,k)$ with $n-2k\geq 2$ and $k\geq 2$. If $\la$ is a JS-partition then
\[\dim\End_{\s_{n-2,2}}(D^\la\da_{\s_{n-2,2}})=2,\quad\dim\End_{\s_{n-4,4}}(D^\la\da_{\s_{n-4,4}})\geq 3.\]
\end{lemma}

\begin{proof}
Since $\la$ is a JS-partition we have that $n-k\equiv k+1\Md 3$. So by assumption $n-k\geq k+4$. Repeated use of Lemmas \ref{l45}, \ref{l39} and \ref{l56} give 
\begin{align*}
D^\la\da_{\s_{n-1}}\cong &D^{(n-k-1,k)},\\
D^\la\da_{\s_{n-2}}\cong &D^{(n-k-2,k)}\oplus D^{(n-k-1,k-1)},\\
D^\la\da_{\s_{n-3}}\sim &(D^{(n-k-2,k-1)}|\ldots|D^{(n-k-3,k)}|\ldots|D^{(n-k-2,k-1)})\oplus D^{(n-k-2,k-1)},\\
D^\la\da_{\s_{n-4}}\sim &(D^{(n-k-2,k-2)}|\ldots|D^{(n-k-4,k)}|\ldots|D^{(n-k-2,k-2)})\oplus D^{(n-k-2,k-2)}\\
&\oplus (D^{(n-k-3,k-1)}|\ldots|D^{(n-k-3,k-1)})\oplus D^{(n-k-3,k-1)}\oplus\ldots.
\end{align*}
So $D^\la\da_{\s_{n-2,2}}$ is semisimple with two non-isomorphic direct summands and then $\dim\End_{\s_{n-2,2}}(D^\la\da_{\s_{n-2,2}})=2$. Further, comparing residues of the removed nodes, it can be checked that $D^{(n-k-2,k-2)}$ and $D^{(n-k-4,k)}$ are in the same block, but $D^{(n-k-3,k-1)}$ is in a distinct block. So $D^\la\da_{\s_{n-4,4}}$ has at least two non-zero block components, at least one of which is not simple. Since $D^\la\da_{\s_{n-4}}$ is self-dual, as are all simple $F\s_{n-4,4}$-modules, we also have that $\dim\End_{\s_{n-4,4}}(D^\la\da_{\s_{n-4,4}})\geq 3$.
\end{proof}

\begin{lemma}\label{l32}
Let $p=3$, $n\equiv 
1\Md 3$ with $n\geq 
10$ and $\la=(n-k,k)$ with $n-2k\geq 2$ and $k\geq 2$. If $\la$ is a JS-partition then there exists $\psi\in\Hom_{\s_n}(M_3,\End_F(D^\la))$ which does not vanish on $S_3$ or there exists $\psi'\in\Hom_{\s_n}(M_4,\End_F(D^\la))$ which does not vanish on $S_4$.
\end{lemma}

\begin{proof}
From Lemma \ref{l36} we have that
\[\dim\End_{\s_{n-4,4}}(D^\la\da_{\s_{n-4,4}})>\dim\End_{\s_{n-2,2}}(D^\la\da_{\s_{n-2,2}}).\]
The lemma then holds by Lemmas \ref{l41} and \ref{l35}.
\end{proof}

\begin{lemma}\label{l29}
Let $p=3$, $n\equiv 0\Md 3$ with $n\geq 9$ and $\la=(n-k,k)$ with $n-2k\geq 2$ and $k\geq 2$. If $\la$ is a JS-partition then the only normal node of $\la$ has residue 1 and $f_1e_1D^\la\cong D^\la|D^{(n-k-1,k,1)}|D^\la$.
\end{lemma}

\begin{proof}
It follows easily from Lemma \ref{L221119} and the assumptions on $n$ and $\la$ that the only normal node of $\la$ has residue 1. So from Lemmas \ref{l45} and \ref{l39}
\[D^\la\otimes M_1\cong f_1e_1D^\la\oplus f_0e_1D^\la\oplus f_2e_1D^\la.\]
Notice that $f_1e_1D^\la\cong f_1D^{\tilde{e}_1(\la)}$ has simple socle and head isomorphic to $D^\la$ from Lemmas \ref{l39} and \ref{l40}.

From Lemma \ref{L160817_0} we have that $M_1\sim D_0|S_1^*$ and that $S_1^*\subseteq M_{1^2}$. So $D^\la\otimes S_1^*\subseteq D^\la\otimes M_{1^2}$. Further
\[D^\la\otimes M_1\sim (\overbrace{D^\la\otimes D_0}^{\cong D^\la})|(D^\la\otimes S_1^*)\]
so that $D^\la\otimes S_1^*\cong (D^\la\otimes M_1)/D$ for some $D\subseteq D^\la\otimes M_1$ with $D\cong D^\la$. Let $B$ be the block component of $D^\la\otimes S_1^*$ corresponding to the block of $D^\la$. Then $B\cong (f_1e_1D^\la)/D^\la$. We will now show that $\soc(B)\cong D^{(n-k-1,k,1)}$. From Lemmas \ref{l45} and \ref{l39} we have that
\[D^\la\da_{\s_{n-2}}\cong\overbrace{D^{(n-k-2,k)}}^{e_0e_1D^\la}\oplus \overbrace{D^{(n-k-1,k-1)}}^{e_2e_1D^\la}.\]
Since $B\subseteq D^\la\otimes S_1^*\subseteq D^\la\otimes M_{1^2}\cong D^\la\da_{\s_{n-2}}\ua^{\s_n}$, comparing blocks we have that the socle of $B$ is contained in the socle of
\[f_1f_0 D^{(n-k-2,k)}\oplus f_0f_1 D^{(n-k-2,k)}\oplus f_1f_2 D^{(n-k-1,k-1)}\oplus f_2f_1 D^{(n-k-1,k-1)}.\]
From Lemma \ref{l40} we have that
\begin{align*}
&\soc(f_0f_1 D^{(n-k-2,k)}\oplus f_1f_2 D^{(n-k-1,k-1)}\oplus f_2f_1 D^{(n-k-1,k-1)})\\
&\cong\soc(f_0 D^{(n-k-2,k,1)}\oplus f_1 D^{(n-k-1,k)}\oplus f_2 D^{(n-k-1,k-1,1)})\\
&\cong D^\la\oplus (D^{(n-k-1,k,1)})^{\oplus 2}.
\end{align*}
Consider now $\soc(f_1f_0 D^\la)$. We have that $f_0D^{(n-k-2,k)}\cong e_0 D^{(n-k-1,k+1)}$ by \cite[Lemma 3.4]{m1}. Thus by Lemmas \ref{l39} and \ref{tsh}
\[f_0D^{(n-k-2,k)}\sim D^{(n-k-1,k)}|C|D^{(n-k-1,k)}\]
for a certain module $C$ such that all composition factors of $C$ are of the form $D^{(n-k-2+3j,k+1-3j)}$ with $j\geq 0$. Let $\mu\in\Par_3(n)$ with $D^\mu\subseteq f_1f_0D^{(n-k-2,k)}$. Then by Lemma \ref{l48}
\[\dim\Hom_{\s_{n-1}}(e_1 D^\mu,f_0D^{(n-k-2,k)})=\dim\Hom_{\s_n}(D^\mu,f_1f_0D^{(n-k-2,k)})\geq 1.\]
By Lemma \ref{l39} there exists a composition factor $D^\nu$ of $f_0D^{(n-k-2,k)}$ such that $\tilde e_1\mu=\nu$ and then, by Lemma \ref{l47}, $\mu=\tilde f_1\nu$. Thus $\mu=\la$ or $\mu=(n-k-2+3j,k+2-3j)$ for some $j\geq 0$. In the second case $e_1D^\mu\cong D^{(n-k-2+3j,k+1-3j)}$, contradicting $\dim\Hom_{\s_{n-1}}(e_1 D^\mu,f_0D^{(n-k-2,k)})\geq 1$ by Lemma \ref{l40}. Thus $\mu=\la$. 

In particular the only simple modules appearing in the socle of 
\[f_1f_0 D^{(n-k-2,k)}\oplus f_0f_1 D^{(n-k-2,k)}\oplus f_1f_2 D^{(n-k-1,k-1)}\oplus f_2f_1 D^{(n-k-1,k-1)}\]
are $D^\la$ and $D^{(n-k-1,k,1)}$. From Lemma \ref{l56} we have that $B$ has exactly one composition factor of the form $D^\la$, one composition factor of the form $D^{(n-k-1,k,1)}$ and possibly other composition factors. So, in view of Lemma \ref{l40}, $\soc(B)\cong D^{(n-k-1,k,1)}$ and then
\[f_1e_1D^\la\sim D^\la|\overbrace{D^{(n-k-1,k,1)}|E|D^\la}^B\]
for a certain module $E$. Since $f_1e_1 D^\la$ and $B$ both have simple socle and $f_1e_1D^\la$ is self-dual (by Lemma \ref{l57}), the lemma follows.
\end{proof}

\begin{lemma}\label{l27}
Let $p=3$, $n\equiv 0\Md 3$ with $n\geq 6$ and $\la=(n-k,k)$ with $n-2k\geq 2$ and $k\geq 2$. If $\la$ is a JS-partition then
\begin{align*}
\dim\Ext^1_{\s_n}(S^{(n-k-1,k,1)},D^\la)=0\quad\text{and}\quad\dim\Ext^1_{\s_n}(D^{(n-k-1,k,1)},D^\la)\leq 1.
\end{align*}
\end{lemma}

\begin{proof}
Let $\mu:=(n-k-1,k,1)$. Since $\mu$ is 3-regular, so that $D^\mu$ is the head of $S^\mu$ we have that $S^\mu\sim\rad(S^\mu)|D^\mu$. So there exists an exact sequence
\[\Hom_{\s_n}(\rad (S^\mu),D^\la)\to\Ext^1_{\s_n}(D^\mu,D^\la)\to\Ext^1_{\s_n}(S^\mu,D^\la).\]
From \cite{jw} we have that $[S^\mu:D^\la]=1$, and then $\dim\Hom_{\s_n}(\rad (S^\mu),D^\la)\leq 1$. It is then enough to prove that $\dim\Ext^1_{\s_n}(S^\mu,D^\la)=0$.

Notice that by assumption $n-k\equiv 2\Md 3$, $k\equiv 1\Md 3$ and $n-2k\geq 4$. In particular $\la$ has no normal node of residue 0, so $e_0D^\la=0$ by Lemma \ref{l39}. Further
\[f_0 S^{(n-k-2,k,1)}\sim \overbrace{S^{(n-k-2,k,1^2)}|S^{(n-k-2,k+1,1)}}^A|S^\mu\]
by \cite[Corollary 17.14]{JamesBook}. Since $(n-k-2,k,1^2)$ and $(n-k-2,k+1,1)$ are 3-regular, we have that $\dim\Hom_{\s_n}(A,D^\la)=0$. Thus there exists an exact sequence
\[0=\Hom_{\s_n}(A,D^\la)\to \Ext^1_{\s_n}(S^\mu,D^\la)\to \Ext^1_{\s_n}(f_0 S^{(n-k-2,k,1)},D^\la).\]
From $e_0D^\la=0$ and \cite[Lemma 1.4]{ks3} it then follows that
\begin{align*}
\dim\Ext^1_{\s_n}(S^{\mu},D^\la)&\leq\dim\Ext^1_{\s_n}(f_0 S^{(n-k-2,k,1)},D^\la)\\
&=\dim\Ext^1_{\s_{n-1}}(S^{(n-k-2,k,1)},e_0 D^\la)\\
&=0.
\end{align*}
\end{proof}

\begin{lemma}\label{l28}
Let $p=3$, $n\equiv 0\Md 3$ with $n\geq 9$ and $\la=(n-k,k)$ with $n-2k\geq 2$ and $k\geq 2$. If $\la$ is a JS-partition then there exists $\psi\in\Hom_{\s_n}(M_3,\End_F(D^\la))$ which does not vanish on $S_3$ or there exists $\psi'\in\Hom_{\s_n}(M_4,\End_F(D^\la))$ which does not vanish on $S_4$.
\end{lemma}

\begin{proof}
In view of Lemma \ref{l41} we may assume that
\[\dim\End_{\s_{n-3,3}}(D^\la\da_{\s_{n-3,3}})=\dim\End_{\s_{n-2,2}}(D^\la\da_{\s_{n-2,2}}).\]
Thus by Lemma \ref{l36}
\[\dim\End_{\s_{n-4,4}}(D^\la\da_{\s_{n-4,4}})>\dim\End_{\s_{n-3,3}}(D^\la\da_{\s_{n-3,3}}).\]
By Lemma \ref{L160817_0} we have that $M_1\sim D_0|S_1^*$ and that $M_4\sim S_4|S_1|A$ for a certain submodule $A\subseteq M_3$ with $M_3/A\cong S_1$. In view of Lemma \ref{l2} it is then enough to prove that
\[\dim\End_{\s_n}(A\oplus S_1,\End_F(D^\la))\leq\dim\End_{\s_{n-3,3}}(D^\la\da_{\s_{n-3,3}}).\]
Since $\la$ is JS we have by Lemma \ref{l46} that
\[\dim\End_{\s_n}(S_1,\End_F(D^\la))=0.\]
From Lemmas \ref{l45} and \ref{l29} we have that $D^\la\otimes M_1\cong D^\la\da_{\s_{n-1}}\ua^{\s_n}\cong B\oplus C$ where $B\cong D^\la|D^{(n-k-1,k,1)}|D^\la$ is the block component of $D^\la\otimes M_1$ of the block containing $D^\la$ and $C$ is the sum of the other block components. It follows from $M_1\sim D_0|S_1^*$ that $D^\la\otimes S_1^*\cong (B/D^\la)\oplus C$. Thus there exists $M\subseteq D^\la\otimes M_3$ with $N\subseteq M\cong (B/D^\la)\oplus C$ such that $N\cong B/D^\la$ and $(D^\la\otimes M_3)/M\cong D^\la\otimes A^*$. Considering block decomposition we then have that
\begin{align*}
\dim\End_{\s_n}(A,\End_F(D^\la))&=\dim\End_{\s_n}(D^\la,D^\la\otimes A^*)\\
&=\dim\End_{\s_n}(D^\la,(D^\la\otimes M_3)/M)\\
&=\dim\End_{\s_n}(D^\la,(D^\la\otimes M_3)/N).
\end{align*}
Since $h(\la)=2<3=p$ we have from Lemma \ref{l27} and \cite[Theorem 2.10]{ks3} that 
\[\dim\Ext^1_{\s_n}(D^{(n-k-1,k,1)},D^\la)\leq 1\quad\text{and}\quad\dim\Ext^1_{\s_n}(D^\la,D^\la)=0.\]
Since $N\cong D^{(n-k-1,k,1)}|D^\la$ it follows that
\begin{align*}
&\dim\End_{\s_n}(A,\End_F(D^\la))\\
&=\dim\End_{\s_n}(D^\la,(D^\la\otimes M_3)/N)\\
&\leq\dim\End_{\s_n}(D^\la,D^\la\otimes M_3)+\dim\Ext^1_{\s_n}(D^{(n-k-1,k,1)},D^\la)\\
&\hspace{11pt}+\dim\Ext^1_{\s_n}(D^\la,D^\la)-1\\
&\leq \dim\End_{\s_n}(D^\la,D^\la\otimes M_3)\\
&=\dim\End_{\s_{n-3,3}}(D^\la\da_{\s_{n-3,3}}).
\end{align*}
So the lemma holds.
\end{proof}

\section{Splitting JS partitions}\label{sJS}

Splitting modules indexed by JS partitions also play a special role in the proof of Theorem \ref{mt}, so they and the corresponding partitions will be studied more in details in this section.

\begin{lemma}\label{L6}
Let $p=2$. If $\la\in\Parinv_2(n)$ is a JS-partition then the parts of $\la$ are odd. Further $n\equiv h(\la)^2\Md 4$.
\end{lemma}

\begin{proof}
Since $\la$ is a JS-partition all parts have the same parity by Lemma \ref{L151119}. It then easily follows that all parts are odd by Lemma \ref{split2}. Let $k$ be maximal with $2k\leq h(\la)$. For $1\leq i\leq k$ we have by Lemma \ref{split2} that $\la_{2i-1}-\la_{2i}=2$ and so $\la_{2i-1}+\la_{2i}\equiv 0\Md 4$ and further if $h(\la)$ is odd then $\la_{h(\la)}=1$. So $n\equiv h(\la)^2\Md 4$.
\end{proof}

\begin{lemma}\label{L11}
Let $p=2$ and $n\geq 6$ be even. Let $\la\in\Parinv_2(n)$ be a JS-partition with $\la\not=\be_n$. Then $n\equiv 0\Md 4$ and $D_2\subseteq\End_F(D^\la)$. Further $D_2\da_{A_n}\subseteq\End_F(E^\la_\pm)$ or $S_3^*\da_{A_n}\subseteq\End_F(E^\la_\pm)$.
\end{lemma}

\begin{proof}
From Lemma \ref{L6} we have that $n\equiv 0\Md 4$. From \cite[Lemma 7.5]{m1} we then have that $D_2\subseteq\End_F(D^\la)$.

From Lemma \ref{L12e} we have that $M_3\cong M_1\oplus A$, where
\[A\cong\overbrace{D_2|D_1|D_3}^{S_3}|\overbrace{D_1|D_2}^{S_2}.\]
From Lemma \ref{LYoung} we have that $A\cong Y_3$ is self-dual, so we also have
\[A\cong\overbrace{D_2|D_1}^{S_2^*}|\overbrace{D_3|D_1|D_2}^{S_3^*}.\]
By Lemma \ref{L1} it then easily follows that $\dim\Hom_{A_n}(A,\End_F(E^\la_\pm))\geq 1$.

From Lemma \ref{Lemma7.1}, by Frobenious reciprocity and since $D_1\cong\hd S_1$,
\begin{align*}
\dim\Hom_{A_n}(D_1\da_{A_n},\End_F(E^\la_\pm))&\leq\dim\Hom_{\s_n}(D_1,\End_F(D^\la))\\
&\leq\dim\Hom_{\s_n}(S_1,\End_F(D^\la))\\
&=0.
\end{align*}
The lemma then follows.
\end{proof}

\begin{lemma}\label{l23}{\cite[Lemma 8.1]{m2}}
Let $p\geq 3$ and $\lambda\in\Parinv_3(n)$ be a JS-partition. Then $n\equiv h(\lambda)^2\Md p$.
\end{lemma}

\section{Spilt-non-split case}\label{sns}

In this section we study irreducible tensor products of the form $E^\lambda_\pm\otimes E^\mu$. We will use $E^\la_\pm$ to refer to $E^\la_+$ or $E^\la_-$, and $E^\la_\mp$ to the other.

For $p\geq 3$ the following lemma holds by \cite[Lemma 3.1]{bk2}.

\begin{lemma}\label{l14}
Let $\la\in\Parinv_p(n)$ and $\mu\in\Par_p(n)\setminus\Parinv_p(n)$. If $E^\lambda_\pm\otimes E^\mu$ is irreducible then
\begin{align*}
\dim\Hom_{\s_n}(\End_F(D^\lambda),\End_F(D^\mu))&\leq 2.
\end{align*}
\end{lemma}

\begin{proof}
Notice that, by Frobenious reciprocity,
\begin{align*}
&\dim\Hom_{\s_n}(\End_F(D^\lambda),\End_F(D^\mu))\\
&=\dim\Hom_{A_n}(\Hom_F(E^\la_\pm,E^\la_+\oplus E^\la_-),\End_F(E^\mu))\\
&=\dim\End_{A_n}(E^\la_\pm\otimes E^\mu)+\dim\Hom_{A_n}(E^\la_\pm\otimes E^\mu,E^\la_\mp\otimes E^\mu).
\end{align*}
The lemma then follows, since $E^\la_+\otimes E^\mu$ and $E^\la_-\otimes E^\mu$ have the same dimension.
\end{proof}

\begin{theor}\label{sns2a}
Let $p=2$, $\la\in\Parinv_2(n)$ and $\mu\in\Par_2(n)\setminus\Parinv_2(n)$. If $E^\la_\pm$ and $E^\mu$ are not 1-dimensional and $E^\la_\pm\otimes E^\mu$ is irreducible, then $\la$ or $\mu$ is equal to $(n-1,1)$ or $\be_n$.
\end{theor}

\begin{proof}
For $n\leq 9$ the theorem holds by comparing dimensions using \cite[Tables]{JamesBook}. So we may assume that $n\geq 10$ and $\la,\mu\not\in\{(n),(n-1,1),\be_n\}$. By Lemma \ref{split2} we then have that $h(\la)\geq 3$.  Note that there always exist $\psi_{0,\la}\in\Hom_{\s_n}(M_0,\End_F(D^\la))$ and $\psi_{0,\mu}\in\Hom_{\s_n}(M_0,\End_F(D^\mu))$ which do not vanish on $S_0$. Further for $2\leq k\leq 3$, from Lemmas \ref{L2} and \ref{L3} there exist $\psi_{k,\la}\in\Hom_{\s_n}(M_k,\End_F(D^\la))$ which do not vanish on $S_k$. Similarly there exists $\psi_{2,\mu}\in\Hom_{\s_n}(M_2,\End_F(D^\mu))$ which does not vanish on $S_2$.

Assume first that $h(\mu)\geq 3$. Then there similarly also exists $\psi_{3,\mu}\in\Hom_{\s_n}(M_3,\End_F(D^\mu))$ which does not vanish on $S_3$. So by Lemma \ref{l15}
\[\dim\Hom_{\s_n}(\End_F(D^\lambda),\End_F(D^\mu))\geq 3,\]
contradicting $E^\la_\pm\otimes E^\mu$ being irreducible by Lemma \ref{l14}.

So we may now assume that $\mu=(n-k,k)$ with $n-2k\geq 3$ and $k\geq 2$. Consider first $n$ odd. Then by Lemma \ref{L4} there exist $\psi_{2,\mu},\psi_{2,\mu}'\in\Hom_{\s_n}(M_2,\End_F(D^\mu))$ with $\psi_{2,\mu}|_{S_2},\psi_{2,\mu}'|_{S_2}$ linearly independent or there exists $\psi_{3,\mu}\in\Hom_{\s_n}(M_3,\End_F(D^\mu))$ which does not vanish on $S_3$. In either case by Lemmas \ref{l15}
\[\dim\Hom_{\s_n}(\End_F(D^\la),\End_F(D^\mu))\geq 3,\]
again leading to a contradiction by Lemma \ref{l14}.

If $n$ is even and $\la$ has at least two normal nodes we can similarly conclude by Lemma \ref{L2} applied $\mu$ and Lemma \ref{L7} or \ref{L8} applied to $\la$.

So assume now that $n$ is even and $\la$ is JS. Then $h(\la)\geq 4$ by Lemma \ref{L6}. So by Lemma \ref{L9} there exists $\psi_{4,\la}\in\Hom_{\s_n}(M_4,\End_F(D^\la))$ which does not vanish on $S_4$. In view of Lemma \ref{L2} (and arguing as above), we may then assume that there does not exist any $\psi_{4,\mu}\in\Hom_{\s_n}(M_4,\End_F(D^\mu))$ which does not vanish on $S_4$. So by Lemma \ref{L10} we may assume that $A\subseteq\End_F(D^\mu)=\End_F(E^\mu)$ with $A\in\{D_2^2,S_3,D_2\oplus D_3\}$. By Lemma \ref{L11} we have that $D_2\subseteq\End_F(D^\la)$. Since $D_0\subseteq\End_F(D^\la),\End_F(D^\mu)$, we may thus assume that $A\in\{S_3,D_2\oplus D_3\}$. From Lemma \ref{L11} we also have that there exists $B\subseteq\End_F(E^\la_\pm)$ with $B\in\{D_2\da_{A_n},S_3^*\da_{A_n}\}$. Further from Lemma \ref{L12e}, $D_2\subseteq S_3$. It then follows that
\[\dim\End_{A_n}(E^\la_\pm\otimes E^\mu)=\dim\Hom_{A_n}(\End_F(E^\la_\pm),\End_F(E^\mu))\geq 2,\]
which also contradicts $E^\la_\pm\otimes E^\mu$ being irreducible.
\end{proof}

\begin{theor}\label{sns2b}
If $p=2$, $n\geq 3$ and $\la\in\Parinv_p(n)$ then $E^\la_\pm\otimes E^{(n-1,1)}$ is irreducible if and only if $n$ is odd and $\la$ is a JS-partition, in which case $E^\la_\pm\otimes E^{(n-1,1)}\cong E^\nu$, where $\nu$ is obtained from $\la$ by removing the top removable node and adding the second bottom addable node.
\end{theor}

\begin{proof}
If $E^\la_\pm\otimes E^{(n-1,1)}$ is irreducible then $D^\la$ is not a composition factor of $D^\la\otimes D_1$ (since $E^{(n-1,1)}$ is not 1-dimensional). In particular, using Lemmas \ref{L12e} and \ref{L12o},
\[[D^\la\otimes M_1:D^\la]=[M_1:D_0]=\left\{\begin{array}{ll}
2,&n\mbox { is even},\\
1,&n\mbox { is odd}.
\end{array}\right.\]
From \cite[Lemma 3.5]{kmt}, Lemma \ref{l45} and block decomposition,
\[[D^\la\otimes M_1:D^\la]=\epsilon_0(\la)(\phi_0(\la)+1)+\epsilon_1(\la)(\phi_1(\la)+1).\]

Assume first that $n$ is even. Then $\la$ has at most two normal nodes. If $\la$ has exactly two normal nodes then we have from \cite[Lemma 6.2]{m1} that $[D^\la\otimes M_1:D^\la]>2$ (notice that $\phi_0(\la)+\phi_1(\la)=3$ by Lemma \ref{l52}). If $\la$ is a JS-partition then the only normal node is the top removable node and the only conormal nodes are the two bottom addable nodes. From Lemma \ref{L6} all these nodes have residue 0, thus $[D^\la\otimes M_1:D^\la]=3$.

So we may now assume that $n$ is odd, in which case it easily follows from $[D^\la\otimes M_1:D^\la]=1$ that $\la$ is a JS-partition. In this case by Lemma \ref{L6} the normal node has residue 0 and the two conormal nodes both have residue 1. Let $A$ be the top removable node of $\la$, $B$ be the second bottom addable node of $\la$ and $C$ be the bottom addable node of $\la$. Then $A$ is the normal node of $\la$ and $B$ and $C$ are the conormal nodes of $\la$. From Lemmas \ref{split2} and \ref{L6} (or Lemma \ref{L151119}) we easily have that $h(\la)\geq 3$. In particular $B$ and $C$ are the two bottom addable nodes of $\tilde{e}_0(\la)=\la\setminus A$. So $B$ and $C$ are conormal in $\tilde{e}_0(\la)$. From Lemma \ref{l47} we have that $A$ is also conormal in $\tilde{e}_0(\la)$. Since $\la$ is a JS-partition it is easy to check that the normal nodes of $\tilde{e}_0(\la)$ are exactly the two top removable nodes. From Lemma \ref{l52} it follows that $A$, $B$ and $C$ are the only conormal nodes of $\tilde{e}_0(\la)$. So, from Lemmas \ref{l45} and \ref{l40},
\[D^\la\otimes M_1
\cong f_0D^{\tilde{e}_0(\la)}\oplus f_1D^{\tilde{e}_0(\la)}\cong D^\la\oplus (\underbrace{D^{(\la\setminus A)\cup B}|\overbrace{\rlap{$\phantom{D^\lambda}$}\ldots\rlap{$\phantom{D^\lambda}$}}^{\text{no }D^{(\la\setminus A)\cup B}}|D^{(\la\setminus A)\cup B}}_{\text{indec. w. simple head and socle}}).
\]
From Lemma \ref{L12o} it then follows that
\[D^\la\otimes D_1\cong \underbrace{D^{(\la\setminus A)\cup B}|\overbrace{\rlap{$\phantom{D^\lambda}$}\ldots\rlap{$\phantom{D^\lambda}$}}^{\text{no }D^{(\la\setminus A)\cup B}}|D^{(\la\setminus A)\cup B}}_{\text{indec. w. simple head and socle}}.
\]
Notice that $\la$ has an odd number of parts, all of which are odd. Since $D^\la\da_{A_n}$ splits, it follows from Lemma \ref{split2} that $\la_{h(\la)}=1$ and then that $D^{(\la\setminus A)\cup B}\da_{A_n}$ does not split (the corresponding partition has an odd number of parts and the last part is 2). Since $\soc(D^\la\otimes D_1)\cong D^{(\la\setminus A)\cup B}$ it follows that
\[\soc((D^\la\otimes D_1)\da_{A_n})\cong\soc((E^\la_+\otimes E^{(n-1,1)})\oplus (E^\la_-\otimes E^{(n-1,1)}))\cong (E^{(\la\setminus A)\cup B})^{\oplus k}\]
for some $k\geq 2$. 
So
\begin{align*}
&\dim\Hom_{\s_n}(D^\la\otimes D_1,E^{(\la\setminus A)\cup B}\ua^{\s_n})\\
&=\dim\Hom_{A_n}((D^\la\otimes D_1)\da_{A_n},E^{(\la\setminus A)\cup B})\\
&\geq 2.
\end{align*}
Since $E^{(\la\setminus A)\cup B}\ua^{\s_n}\cong D^{(\la\setminus A)\cup B}|D^{(\la\setminus A)\cup B}$ and the socle of $D^\la\otimes D_1$ is simple, we have that $E^{(\la\setminus A)\cup B}\ua^{\s_n}\subseteq D^\la\otimes D_1$ and then that $D^\la\otimes D_1\cong D^{(\la\setminus A)\cup B}|D^{(\la\setminus A)\cup B}$, 
from which the theorem follows.
\end{proof}

\begin{theor}\label{sns3}
Let $p=3$, $\la\in\Parinv_3(n)$ and $\mu\in\Par_3(n)\setminus\Parinv_3(n)$. If $E^\la_\pm$ and $E^\mu$ are not 1-dimensional then $E^\lambda_\pm\otimes E^\mu$ is irreducible if and only if $\mu\in\{(n-1,1),(n-1,1)^\Mull\}$, $\la$ is a JS-partition and $n\not\equiv 0\Md 3$. In this case $E^\lambda_\pm\otimes E^{(n-1,1)}\cong E^\nu$, where $\nu$ is obtained from $\lambda$ by removing the top removable node and adding the bottom addable node.
\end{theor}

\begin{proof}
If $\mu\in\{(n-1,1),(n-1,1)^\Mull\}$ the theorem holds by \cite[Theorem 3.3]{bk2} and Lemma \ref{l23}. So we may now assume that $\mu\not\in\{(n),(n)^\Mull,(n-1,1),(n-1,1)^\Mull\}$. For $n\leq 9$ the theorem can be checked separately, using \cite[Tables]{JamesBook}. So we may also assume that $n\geq 10$. By \cite[Lemma 2.2]{bkz}, and checking small cases separately, it then follows that $\al>\al^\Mull$ for $\al\in\{(n),(n-2,2),(n-3,3),(n-4,2^2)\}$.

From \cite[Theorem 9.2]{m2} we may assume that $\la$ is a JS-partition. From Lemma \ref{Mull} we have that $h(\la)\geq 3$. Assume first that $h(\mu),h(\mu^\Mull)\geq 3$. Then for $k=0$ and $k=3$ (the second case by Lemmas \ref{psi3} and \ref{l30}) there exist $\psi_{k,\la}\in\Hom_{A_n}(M_k,\End_F(E^\la_\pm))$ and $\psi_{k,\mu}\in\Hom_{A_n}(M_k,\End_F(E^\mu))$ which do not vanish on $S_k$. By Lemma \ref{l15}, 
\[\dim\End_{A_n}(E^\la_\pm\otimes E^\mu)=\dim\Hom_{A_n}(\End_F(E^\la_\pm),\End_F(E^\mu))\geq 2,\]
contradicting $E^\la_\pm\otimes E^\mu$ being irreducible.

So, up to exchange of $\mu$ and $\mu^\Mull$, we may assume that $\mu=(n-k,k)$ with $k\geq 2$ and $n-2k\geq 2$. If the removable nodes of $\mu$ have distinct residues then apply Lemmas \ref{psi2} and \ref{psi3} to $\la$ and Lemmas \ref{psi2} and \ref{l38} to $\mu$. If the removable nodes of $\mu$ have the same residue apply Lemma \ref{psi2} to $\la$ and Lemma \ref{l37} to $\mu$. Similarly to the above case we then have by Lemma \ref{l15} that in either case
\[\dim\Hom_{\s_n}(\End_F(D^\lambda),\End_F(D^\mu))\geq 3.\]
again contradicting $E^\la_\pm\otimes E^\mu$ being irreducible, due to Lemma \ref{l14}.

So we may assume that $\mu$ is also a JS-partition. From Lemma \ref{l23} we have that $n\equiv h(\la)^2\equiv 0\mbox{ or }1\Md 3$. If $n\equiv 1\Md 3$, then $h(\la)\geq 4$ by Lemmas \ref{Mull} and \ref{l23}. In this case apply Lemmas \ref{psi2}, \ref{psi3} and \ref{l33} to $\la$ and Lemmas \ref{psi2} and \ref{l32} to $\mu$. If $n\equiv 0\Md 3$ and $h(\la)>3$ apply Lemmas \ref{psi2}, \ref{psi3} and \ref{l33} to $\la$ and Lemmas \ref{psi2} and \ref{l28} to $\mu$. If $n\equiv 0\Md 3$ and $h(\la)=3$ then apply Lemmas \ref{psi2} and \ref{l8} to $\la$ and Lemmas \ref{psi2} and \ref{l9} to $\mu$. In each of these cases we then again contradict $E^\la_\pm\otimes E^\mu$ being irreducible by Lemmas \ref{l15} and \ref{l14}.
\end{proof}

\section{Double split case}\label{ds}

In this section we study irreducible tensor products of the form $E^\la_\pm\otimes E^\mu_\pm$.

\begin{theor}\label{ds2}
Let $p=2$. If $\la,\mu\in\Parinv_2(n)$ and $E^\la_\pm\otimes E^\mu_\pm$ is irreducible, then $n\not\equiv 2\Md 4$ and $\la=\be_n$ or $\mu=\be_n$.
\end{theor}

\begin{proof}
For $n\leq 8$ the theorem can be checked separately. So we may assume $n\geq 9$. By Lemma \ref{split2} we then have that $(n-3,3)\in\Par_2(n)\setminus\Parinv_2(n)$ and that if $\la,\mu\not=\be_n$ then $h(\la),h(\mu)\geq 3$. In this case by Lemmas \ref{l15} and \ref{L1}
\[\dim\End_{A_n}(E^\la_\pm\otimes E^\mu_\pm)=\dim\Hom_{A_n}(\End_F(E^\la_\pm),\End_F(E^\mu_\pm))\geq 2\]
(similar to the proofs of Theorems \ref{sns2a} and \ref{sns3}), contradicting $E^\la_\pm\otimes E^\mu_\pm$ being irreducible.

So $\la=\be_n$ or $\mu=\be_n$, and then $n\not\equiv 2\Md 4$ by Lemma \ref{split2}.
\end{proof}

\begin{theor}\label{ds3}
Let $p=3$ and $\lambda,\mu\in\Parinv_3(n)$ and assume that $E^\la_\pm$ and $E^\mu_\pm$ are not 1-dimensional. Then $E^\lambda_\pm\otimes E^\mu_\pm$ is irreducible if and only if, up to exchange, $E^\la_\pm=E^{(4,1,1)}_+$ and $E^\mu_\pm=E^{(4,1,1)}_-$. Further $E^{(4,1,1)}_+\otimes E^{(4,1,1)}_-\cong E^{(4,2)}$.
\end{theor}

\begin{proof}
For $n\leq 8$ it can be proved using \cite[Tables]{JamesBook} that if $E^\la_\pm\otimes E^\mu_\pm$ is irreducible then $n=6$ and $\la,\mu=(4,1,1)$, in which case the theorem can be checked using \cite{Atl}. So we may now assume that $n\geq 9$. Then $(n-3,3)>(n-3,3)^\Mull$ by \cite[Lemma 2.2]{bkz} and so by Lemmas \ref{l15} and \ref{l30},
\[\dim\End_{A_n}(E^\la_\pm\otimes E^\mu_\pm)=\dim\Hom_{A_n}(\End_F(E^\la_\pm),\End_F(E^\mu_\pm))\geq 2.\]
In particular $E^\la_\pm\otimes E^\mu_\pm$ is not irreducible.
\end{proof}

\section{Proof of Theorem \ref{mt}}\label{s3}

We will now prove our main result. We will consider the cases $p=2$ and $p\geq 3$ separately.

{\bf Case 1:} $p=2$.

If $\la,\mu\in\Par_2(n)\setminus\Parinv_2(n)$ and $E^\la\otimes E^\mu$ is irreducible as $FA_n$-module, then $D^\la\otimes D^\mu$ is irreducible as $F\s_n$-module. If $E^\la$ and $E^\mu$ are not 1-dimensional then by \cite[Main Theorem]{bk} and \cite[Theorems 1.1 and 1.2]{m1} we have that $n\equiv 2\Md 4$ and $D^\la\otimes D^\mu\cong D^\nu$ with $\nu=(n/2-j,n/2-j-1,j+1,j)$ with $0\leq j\leq (n-6)/4$. By Lemma \ref{split2} it follows that $\nu\in\Parinv_2(n)$, contradicting $E^\la\otimes E^\mu$ being irreducible. If $\la\in\Parinv_2(n)$ and $\mu\in\Par_2(n)\setminus\Parinv_2(n)$ the theorem holds by Theorems \ref{sns2a} and \ref{sns2b}. If $\la,\mu\in\Parinv_2(n)$ the theorem holds by Theorem \ref{ds2}.

{\bf Case 2:} $p\geq 3$.

Note that from Lemma \ref{l23} if $\la\in\Parinv_p(n)$ is JS, then $n\equiv h(\la)^2\Md p$. In particular in this case $n\equiv 0\Md p$ if and only if $h(\la)\equiv 0\Md p$. Assume that $\la\in\Parinv_p(n)$ is JS and that $n\not\equiv 0\Md p$. Let $A$ be the top removable node of $\la$ and $B$ and $C$ be the two bottom addable nodes of $\la$. Then $A$ is the only normal node of $\la$ and $B$ and $C$ are the only conormal nodes of $\la$. Since $h(\la)\not\equiv 0\Md p$, the bottom addable node of $\la$ has residue different from 0. In view of Lemma \ref{l17} we then have that $\res(A)=0$ and that $\res(B)=i=-\res(C)$ for some residue $i\not=0$. By \cite[Lemma 2.9]{bk2} we further have that $A,B,C$ are the only conormal nodes of $\la\setminus A$. Comparing residues we have that $(\la\setminus A)^\Mull=\la\setminus A$ and that $((\la\setminus A)\cup B)^\Mull=(\la\setminus A)\cup C$ by Lemma \ref{l17}. So $(\la\setminus A)\cup B,(\la\setminus A)\cup C\in\Par_p(n)\setminus\Parinv_p(n)$ and $E^{(\la\setminus A)\cup B}\cong E^{(\la\setminus A)\cup C}$. For $p\geq 5$ the theorem then holds by \cite[Main Theorem]{bk2} and \cite[Theorem 1.1]{m2}. So assume now that $p=3$. If $\la,\mu\in\Par_3(n)\setminus\Parinv_3(n)$ and $E^\la$ and $E^\mu$ are not 1-dimensional, then $E^\la\otimes E^\mu$ is not irreducible by \cite[Main Theorem]{bk}. If $\la\in\Parinv_3(n)$ and $\mu\in\Par_3(n)\setminus\Parinv_3(n)$ the theorem holds by Theorem \ref{sns3} and the above observation. If $\la,\mu\in\Parinv_3(n)$ the theorem holds by Theorem \ref{ds3}.

\section{Tensor products with basic spin}\label{s2}

In this section we give some restrictions on tensor products with basic spin module in characteristic 2 which might be irreducible.

\begin{lemma}\label{L13}
Let $p=2$ and $\la,\nu\in\Par_2(n)$. If $[D^\la\otimes D^{\be_n}:D^\nu]=2^ib$ with $b$ odd then $h(\nu)\leq 4i+2$ if $n$ is odd or $h(\nu)\leq 4i+4$ if $n$ is even. Further if $D^\nu\subseteq D^\la\otimes D^{\be_n}$ then $h(\la)\leq 2h(\nu)$.
\end{lemma}

\begin{proof}
For $\gamma\in\Par(n)$ let $\xi^\gamma$ be the Brauer character of $M^\gamma$. For $\psi\in\Par_2(n)$ let $\phi^\psi$ be the Brauer character of $D^\psi$. If $\alpha\in\Par(n)$ is the cycle partition of a 2-regular conjugacy class and $\phi$ is any Brauer character of $\s_n$, let $\phi_\alpha$ be the value that $\phi$ takes on the conjugacy class indexed by $\alpha$.

Let $c:=2i+1$ if $n$ is odd or $c:=2i+2$ if $n$ is even. Let $\al\in\Par(n)$ correspond to a 2-regular conjugacy class of $\s_n$. We have that $\phi^{\be_n}_{\al}=\pm 2^{\lfloor (h(\al)-1)/2\rfloor}$ by \cite[VII, p.203]{s5}. In particular if  $\phi^{\be_n}_\al$ is not divisible by $2^{i+1}$ then $h(\al)\leq c$ (note that $h(\al)\equiv n\Md 2$ since $\al$ is the cycle partition of a 2-regular conjugacy class).

For $\gamma\in\Par(n)$ and $1\leq j\leq n$ let $a_j=a_j(\gamma)$ be the number of parts of $\gamma$ equal to $j$. Further let $A=A(\gamma):=a_1!\cdots a_n!$ and $\overline{A}=\overline{A}(\gamma)$ be the largest power of 2 dividing $A$. Since $M^\la=1\ua_{\s_{\la}}^{\times_j\s_j\wr \s_{a_j}}\ua^{\s_n}$ we have that $A\mid \xi^\gamma_\alpha$ for each $\al\in\Par(n)$ corresponding to a 2-regular conjugacy class. Since irreducible Brauer characters are linearly independent modulo 2 (see for example \cite[Theorem 15.5]{isaacs}), we then have that $\xi^\gamma=\overline{A}\,\overline{\xi}^\gamma$ with $\overline{\xi}^\gamma$ a Brauer character. Further, whenever they are defined, $\xi^\gamma_\gamma=A$ (and so $\overline{\xi}^\gamma_\gamma$ is odd) and $\xi^\gamma_\psi=0$ if $h(\psi)\leq h(\gamma)$ and $\psi\not=\gamma$. In particular there exists $b_\gamma\in\N$ such that if
\[\phi=\phi^{\be_n}(\phi^\la+\sum_{\gamma:h(\gamma)\leq c}b_\gamma\overline{\xi}^\gamma)\]
then $\phi_\al$ is divisible by $2^{i+1}$ for each $\al\in\Par(n)$ corresponding to a 2-regular conjugacy class (start by choosing $b_{(n)}$ so that this holds for $\al=(n)$, if $n$ is odd, then consider $b_\al$ for partitions $\al$ with two parts and so on).

Again since irreducible Brauer characters are linearly independent modulo 2, it follows that $\phi=2^{i+1}\overline{\phi}$ for some Brauer character $\overline{\phi}$. If $m$ is the multiplicity of $\phi^\nu$ in $\overline{\phi}$ then
\begin{align*}
m&=1/2^{i+1}([D^{\be_n}\otimes D^\la:D^\nu]+\sum_{\gamma:h(\gamma)\leq c}b_\gamma/\overline{A}(\gamma)[D^{\be_n}\otimes M^\gamma:D^\nu])\\
&=b/2+\sum_{\gamma:h(\gamma)\leq c}b_\gamma/(2^{i+1}\overline{A}(\gamma))[D^{\be_n}\otimes M^\gamma:D^\nu].
\end{align*}
Since $b$ is odd and $m\in\N$, there then exists $\gamma\in\Par(n)$ with $h(\gamma)\leq c$ such that
\[[S^{\be_n}\otimes M^\gamma:D^\nu]\geq[D^{\be_n}\otimes M^\gamma:D^\nu]\geq 1.\]
Note that $S^{\be_n}\otimes M^\gamma\cong S^{\be_n}\da_{\s_\gamma}\ua^{\s_n}$. In view of the Littlewood-Richardson rule, in characteristic 0, any composition factor of $S^{\be_n}\da_{\s_{\gamma}}$ is of the form $S^{\al^1}\otimes\ldots\otimes S^{\al^{h(\gamma)}}$ with $\al^j\in\Par(\gamma_j)$ such that $h(\al^j)\leq h(\be_n)=2$ and then any composition factor of $S^{\be_n}\otimes M^\gamma$ is of the form $S^\al$ with $h(\al)\leq h(\gamma)h(\be_n)\leq 2c$. Considering reduction modulo 2 we then have that any composition factor of $S^{\be_n}\otimes M^\gamma$, and so in particular also any composition factor of $D^{\be_n}\otimes M^\gamma$, is of the form $D^\zeta$, where $\zeta\in\Par_2(n)$ has at most $2c$ parts. It follows that $h(\nu)\leq 2c$.

Assume now that $D^\nu\subseteq D^\la\otimes D^{\be_n}$. Since
\[\dim\Hom_{\s_n}(D^\la,D^{\be_n}\otimes D^\nu)=\dim\Hom_{\s_n}(D^\nu,D^\la\otimes D^{\be_n})\geq 1\]
it follows that
\[[S^{\be_n}\otimes M^\nu:D^\la]\geq[D^{\be_n}\otimes D^\nu:D^\la]\geq 1.\]
So, similarly to the above, $h(\la)\leq h(\be_n)h(\nu)=2h(\nu)$.
\end{proof}

\begin{theor}
Let $p=2$, $\la\in\Par_2(n)$ and assume that $D^\la$ and $D^{\be_n}$ are not 1-dimensional and that exactly one of them splits when restricted to $A_n$. Then $E^{\be_n}_\pm\otimes E^\la$ or $E^\la_\pm\otimes E^{\be_n}$ is irreducible if and only if $D^\la\otimes D^{\be_n}\sim D^\nu|D^\nu$ with $\nu\in\Par_2(n)\setminus\Parinv_2(n)$. In this case $h(\nu)\leq 6$ if $n$ is odd, $h(\nu)\leq 8$ if $n$ is even and $h(\la)\leq 2h(\nu)$. Further $\la$ has at most 2 normal nodes if $n$ is odd or at most 3 normal nodes if $n$ is even.
\end{theor}

\begin{proof}
From \cite[Main Theorem]{bk} and \cite[Theorems 1.1 and 1.2]{m1} if $D^\la\otimes D^{\be_n}$ is simple as $F\s_n$-module, then $n\equiv 2\Md 4$ and $h(\la)=h(\be_n)=2$. So from Lemma \ref{split2} neither $D^\la$ nor $D^{\be_n}$ splits in this case. Thus we may assume that $D^\la\otimes D^{\be_n}$ is not simple as $F\s_n$-module. Let $\{\al,\ga\}=\{\la,\be_n\}$ such that $\al\in\Parinv_2(n)$ and $\ga\not\in\Parinv_2(n)$. Then $E^\al_+\otimes E^\ga\cong(E^\al_-\otimes E^\ga)^\sigma$ with $\sigma\in\s_n\setminus A_n$. In particular $E^\al_+\otimes E^\ga$ is irreducible if and only if $E^\al_-\otimes E^\ga$ is irreducible. So, since $D^\la\otimes D^{\be_n}$ is not simple as $F\s_n$-module, $E^\al_\pm\otimes E^\ga$ is irreducible if and only if $D^\la\otimes D^{\be_n}\sim D^\nu|D^\nu$ with $\nu\in\Par_2(n)\setminus\Parinv_2(n)$. In this case $h(\nu)\leq 6$ if $n$ is odd, $h(\nu)\leq 8$ if $n$ is even and $h(\la)\leq 2h(\nu)$ by Lemma \ref{L13}.

If $n$ is odd then $M_1\cong D_0\oplus D_1$ by Lemma \ref{L12o}. Since $\be_n$ is not a JS-partition in this case, we have that $D_1\subseteq\End_F(D^{\be_n})$ by Lemma \ref{l53}. If $\la$ has at least 3 normal nodes then $D_1^{\oplus 2}\subseteq \End_F(D^\la)$ from Lemma \ref{l53}.

If $n$ is even then $M_1\cong D_0|D_1|D_0\sim D_0|S_1^*$ by Lemma \ref{L12e} and self-duality of $M_1$. From Lemma \ref{Lemma7.1} we also have that $D_1$ or $S_1$ is contained in $\End_F(D^{\be_n})$. If $\la$ has at least 4 normal nodes we have from Lemma \ref{l53} that
\begin{align*}
&\dim\Hom_{\s_n}(S_1^*,\End_F(D^\la))\\
&\geq\dim\Hom_{\s_n}(M_1,\End_F(D^\la))-\dim\Hom_{\s_n}(D_0,\End_F(D^\la))\\
&=\dim\End_{\s_{n-1}}(D^\la\da_{\s_{n-1}})-\dim\End_{\s_n}(D^\la)\\
&\geq 3.
\end{align*}
Since $S_1^*\cong D_1|D_0$ and $\dim\Hom_{\s_n}(D_0,\End_F(D^\la))=1$, we then have that $(S_1^*)^{\oplus 2}\subseteq \End_F(D^\la)$.

From $D_0\subseteq\End_F(D^{\be_n}),\End_F(D^\la)$, it follows that in either case
\[\dim\Hom_{\s_n}(\End_F(D^\la),\End_F(D^{\be_n}))\geq 3\]
and so $E^\al_\pm\otimes E^\ga$ is not irreducible by Lemma \ref{l14}.
\end{proof}

\begin{theor}
Let $p=2$, $n\not\equiv 2\Md 4$, $\la\in\Parinv_2(n)$ and $\epsilon,\delta,\epsilon',\delta'\in\{\pm\}$. If $E^\la_\pm$ and $E^{\be_n}_\pm$ are not 1-dimensional and $E^\la_\epsilon\otimes E^{\be_n}_\delta$ is irreducible then one of the following holds:
\begin{itemize}
\item
$D^\la\otimes D^{\be_n}\sim D^\nu|D^\nu|D^\nu|D^\nu$ with $\nu\in\Par_2(n)\setminus\Parinv_2(n)$. In this case $E^\la_{\epsilon'}\otimes E^\mu_{\delta'}\cong E^\nu$ is irreducible and $h(\nu)\leq 10$ if $n$ is odd or $h(\nu)\leq 12$ if $n$ is even.

\item
$D^\la\otimes D^{\be_n}\sim D^\nu|D^\nu$ with $\nu\in\Parinv_2(n)$. In this case $E^\la_{\epsilon'}\otimes E^\mu_{\delta'}\in\{E^\nu_+,E^\la_-\}$ is irreducible and $h(\nu)\leq 6$ if $n$ is odd or $h(\nu)\leq 8$ if $n$ is even.

\item
$[D^\la\otimes D^{\be_n}:D^\nu]=2$ with $\nu\in\Par_2(n)\setminus\Parinv_2(n)$ and $E^\la_\epsilon\otimes E^{\be_n}_\delta\cong E^\la_{-\epsilon}\otimes E^{\be_n}_{-\delta}\cong E^\nu$, while $E^\la_{-\epsilon}\otimes E^{\be_n}_\delta\not\cong E^\nu\not\cong E^\la_{\epsilon}\otimes E^{\be_n}_{-\delta}$. Further $h(\nu)\leq 6$ if $n$ is odd or $h(\nu)\leq 8$ if $n$ is even.

\item
$n\equiv 0\Md 4$, $[D^\la\otimes D^{\be_n}:D^\nu]=1$ with $\nu\in\Parinv_2(n)$, $\{E^\la_\epsilon\otimes E^{\be_n}_\delta,E^\la_{-\epsilon}\otimes E^{\be_n}_{-\delta}\}=\{E^\nu_+,E^\nu_-\}$, while $E^\la_{-\epsilon}\otimes E^{\be_n}_\delta,E^\la_{\epsilon}\otimes E^{\be_n}_{-\delta}\not\in\{E^\nu_+,E^\nu_-\}$. Further $h(\nu)\leq 4$.
\end{itemize}
In each of the above cases $h(\la)\leq 2h(\nu)$. Further $\la$ has at most 3 normal nodes if $n$ is odd or at most 4 normal nodes if $n$ is even.
\end{theor}

\begin{proof}
Note that if $\sigma\in\s_n\setminus A_n$ then $E^\la_{\epsilon'}\otimes E^{\be_n}_{\delta'}\cong(E^\la_{-\epsilon'}\otimes E^{\be_n}_{-\delta'})^\sigma$.

In particular if $E^\la_\epsilon\otimes E^{\be_n}_\delta\cong E^\nu$ then $E^\la_{-\epsilon}\otimes E^{\be_n}_{-\delta}\cong E^\nu$ and either both or neither of $E^\la_\epsilon\otimes E^{\be_n}_{-\delta}$ and $E^\la_{-\epsilon}\otimes E^{\be_n}_\delta$ is isomorphic to $E^\nu$. Similarly if $E^\la_\epsilon\otimes E^{\be_n}_\delta\cong E^\nu_\pm$ then $E^\la_{-\epsilon}\otimes E^{\be_n}_{-\delta}\cong E^\nu_\mp$ and $\{E^\la_\epsilon\otimes E^{\be_n}_{-\delta},E^\la_{-\epsilon}\otimes E^{\be_n}_\delta\}$ is either equal to or disjoint from $\{E^\nu_+,E^\nu_-\}$.

So we are in one of the following cases:
\begin{enumerate}
\item
$D^\la\otimes D^{\be_n}\sim D^\nu|D^\nu|D^\nu|D^\nu$ with $\nu\in\Par_2(n)\setminus\Parinv_2(n)$ and $E^\la_{\epsilon'}\otimes E^\mu_{\delta'}\cong E^\nu$.

\item
$D^\la\otimes D^{\be_n}\sim D^\nu|D^\nu$ with $\nu\in\Parinv_2(n)$ and $E^\la_{\epsilon'}\otimes E^\mu_{\delta'}\in\{E^\nu_+,E^\la_-\}$.

\item
$[D^\la\otimes D^{\be_n}:D^\nu]=2$ with $\nu\in\Par_2(n)\setminus\Parinv_2(n)$ and $E^\la_\epsilon\otimes E^{\be_n}_\delta\cong E^\la_{-\epsilon}\otimes E^{\be_n}_{-\delta}\cong E^\nu$, while $E^\la_{-\epsilon}\otimes E^{\be_n}_\delta\not\cong E^\nu\not\cong E^\la_{\epsilon}\otimes E^{\be_n}_{-\delta}$.

\item
$[D^\la\otimes D^{\be_n}:D^\nu]=1$ with $\nu\in\Parinv_2(n)$, $\{E^\la_\epsilon\otimes E^{\be_n}_\delta,E^\la_{-\epsilon}\otimes E^{\be_n}_{-\delta}\}=\{E^\nu_+,E^\nu_-\}$, while $E^\la_{-\epsilon}\otimes E^{\be_n}_\delta,E^\la_{\epsilon}\otimes E^{\be_n}_{-\delta}\not\in\{E^\nu_+,E^\nu_-\}$.
\end{enumerate}

If $[D^\la\otimes D^{\be_n}:D^\nu]=2^i$ then, from Lemma \ref{L13}, $h(\la)\leq 2 h(\nu)$ and that $h(\nu)\leq 4i+2$ if $n$ is odd or $h(\nu)\leq 4i+4$  if $n$ is even (note that we always have $D^\nu\subseteq D^\la\otimes D^{\be_n}$, since $E^\nu_{(\pm)}\cong E^\la_\eps\otimes E^{\be_n}_\de\subseteq(D^\la\otimes D^{\be_n})\da_{A_n}$). In case (iv) if $n$ is odd then $h(\nu)\leq 2$ and so $\nu=\be_n$ by Lemma \ref{split2}, contradicting $E^\la_\pm$ not being 1-dimensional.

This proves the theorem, up to the bound on the number of normal nodes of $\la$. Notice that  if $n$ is odd then $M_1\cong D_0\oplus D_1$, while if $n$ is even then $M_1\cong D_0|D_1|D_0$ by Lemma \ref{L12e}. If $n$ is odd then $D_1\subseteq\End_F(D^{\be_n})$ since in this case $\be_n$ is not a JS-partition. If $n$ is even then $n\equiv 0\Md 4$ by Lemma \ref{split2} and so $D_1\subseteq\End_F(D^{\be_n})$ from Lemma \ref{Lemma7.1}. If $\la$ has at least 4 normal nodes if $n$ is odd or at least 5 normal nodes if $n$ is even then $D_1^{\oplus 3}\subseteq\End_F(D^\la)$. It then follows that there exist $\epsilon'',\delta''\in\{\pm\}$ such that
\begin{align*}
D_1&\subseteq\Hom_F(E^{\be_n}_\delta,E^{\be_n}_{\delta''}),\\
D_1^{\oplus 2}&\subseteq\Hom_F(E^\la_{\epsilon''},E^\la_\epsilon),
\end{align*}
and so $E^\la_\eps\otimes E^{\be_n}_\de$ is not irreducible by \cite[Lemma 3.4]{bk2}.
\end{proof}

\section{Acknowledgements}

The author thank Alexander Kleshchev for some comments and pointing out some references. The author also thanks the referee for comments.

The author was supported by the DFG grant MO 3377/1-1.

\end{document}